\documentclass[12pt]{article}
\usepackage{amssymb}
\usepackage{amsmath}
\usepackage{amsthm}

\numberwithin{equation}{section}
\usepackage{graphicx}
\newtheorem{theorem}{Theorem}[section]
\newtheorem{lemma}{Lemma}[section]
\newtheorem{proposition}{Proposition}[section]
\newtheorem{definition}{Definition}[section]

\newtheorem{remark}{Remark}[section]
\usepackage{epstopdf}
\usepackage{float}
\usepackage{epsf,epsfig,subfigure}

\newcommand{\f}{\frac}
\newcommand{\R}{\mathbb {R}}
\newcommand{\Z}{\mathbb {Z}}

\newcommand{\mc}{\mathcal}

\usepackage{color}

\date{}
\begin{document}

	\title{\bf Traveling waves and spreading speeds for time-space periodic
monotone systems\thanks{The research leading to these results has received funding from the European Research Council under the European UnionÕs Seventh Framework Programme
(FP/2007-2013) / ERC Grant Agreement n.321186 - ReaDi - Reaction-Diffusion Equations, Propagation and
Modelling. J. Fang's research is supported in part by the National Natural Science Foundations of China and the Fundamental Research Funds for Central Universities. X.-Q. Zhao's research is supported in part by the Natural Science and Engineering Research Council of Canada. Corresponding author: zhao@mun.ca (X.-Q. Zhao).}}

\author{{Jian Fang $^\text{a,b}$, Xiao Yu $^\text{c}$ and Xiao-Qiang Zhao $^\text{c}$}
\vspace{2mm}\\
{\small $^\text{a}$   Ecole des Hautes Etudes en Sciences Sociales, CAMS,}\\
{\small 75244 Paris, Cedex 13, France}\\
{\small $^\text{b}$ Department of Mathematics, Harbin Institute of Technology,}\\
{\small Harbin, 150001, China}\\{\small $^\text{c}$ Department of Mathematics and Statistics, Memorial University of Newfoundland,}\\
{\small St. John's, NL A1C  5S7, Canada}}

\maketitle

\begin{abstract}
The theory of traveling waves and spreading speeds is developed for time-space periodic monotone semiflows with monostable structure. By using traveling waves of the associated Poincar\'e maps in a strong sense, we establish the existence of time-space periodic traveling waves and spreading speeds.
We then apply these abstract results to a two species competition reaction-advection-diffusion
model. It turns out that the minimal wave speed exists and coincides with the single spreading
speed for such a system no matter whether the spreading speed is linearly determinate.
We also obtain a set of sufficient conditions for the spreading speed to be linearly
determinate.
\end{abstract}

{\bf Key words:} Monotone semiflows, time-space periodicity, traveling waves, spreading speeds, linear determinacy, competitive systems.

\smallskip
	
{\bf AMS Subject Classification: }  35C07, 35K57, 37C65, 37N25, 92D25

\section{Introduction}

Periodic environment of space and/or time is one of the useful approximations to understand the influence of the environmental heterogeneity on the propagation phenomena arising from ecological and biological processes. The study of reaction-diffusion equations in space-periodic media goes back to Freidlin \cite{Fri} and Freidlin and Gartner\cite{GF}. The notion of pulsating traveling fronts was introduced by Shigesada, Kawasaki and Teramoto \cite{SKT}. The front solution having the exact form $u(t,x)=U(x,x-ct)$ was found and constructed by Xin \cite{Xin91}. It is also called a periodically varying wavefront by Hudson and Zinner \cite{HZ95}. These two notions of fronts are equivalent when the speed is not zero. Berestycki and Hamel \cite{BH02CPAM} established various existence, uniqueness and monotonicity of pulsating fronts for a general reaction-diffusion equation with combustion-type nonlinearity or monostable nonlinearity. Hamel and Roques \cite{HR11} gave a complete classification of all KPP pulsating fronts and obtained some global stability properties for the fronts including the one with minimal speed. Weinberger \cite{Wein02} proved the existence of the minimal wave speed and its coincidence with the spreading speed for a recursion defined by an order-preserving compact operator of monostable type, without assuming a KPP type condition. For a monostable semiflows in one-dimensional periodic environment, Liang and Zhao \cite{Liang2} introduced a topologically conjugate semiflow defined in spatially discrete homogeneous environment and then showed that the spreading speed exists and coincides with the minimal wave speed of the front having the form $U(x,x-ct)$. For monotone semiflows of bistable type, Fang and Zhao \cite{FZJEMS}  interpreted the bistability from a monotone dynamical system point of view to find a link between the monostable subsystems and bistable system itself, which is
used to establish the existence of bistable wavefronts.
We refer to \cite{KRS, CDM, Shen1, Shen2} for nonlocal dispersal equations, and two survey papers \cite{Xin, HR14} for more references. There are also quite a few
investigations on time-periodic fronts of reaction-diffusion equations, see, e.g.,
 \cite{ABC99,BaoWangJDE13,Fre, MZ, ZR,ZhaoRuanJDE14,ZZ}
and references therein.
 For time-periodic semiflows in one dimensional continuous medium, Liang, Yi and Zhao \cite{LYZ} used the wavefront $W(x-c\omega)$  obtained for the Poincar\'e map $Q_\omega$ to construct a two-variable function $U(t,\xi):=Q_t[W](\xi+ct)$, which is then shown to be a time-periodic traveling wave for the semiflow.  However, when the medium is discrete, say $\Z$ for instance, such a construction may not give rise to a traveling wave since $Q_t[W](\xi)$ is not well defined for all $\xi\in \R$.
We will obtain traveling waves in a strong sense for the associated Poincar\'e map so that this evolution approach is still applicable.

For reaction-diffusion equations with time-space periodicity in $\R^N$:
\begin{equation}
\partial_t u-\nabla \cdot (A(t,x)\nabla u)+q(t,x)\cdot \nabla u=f(t,x,u),
\end{equation}
Nadin \cite{Nadin-JMPA} introduced the following definition of pulsating traveling fronts:
\begin{definition}\label{wave-Nadin}
A function $u(t,x)$ is a pulsating traveling front of speed $c$ in the direction $-e$ that connects $p^-$ to $p^+$ if it can be written as $u(t,x)=\phi(x\cdot e+ct,t,x)$, where $\phi\in L^\infty(\R\times\R\times\R^N)$ is such that for almost every $y\in\R$, the function $(t,x)\mapsto \phi(y+x\cdot e+ct,t,x)$ satisfies the above equation.  The function $\phi$ is requested
 to be periodic in its second and third variables and to satisfy
\begin{equation}
\begin{cases}
\phi(z,t,x)-p^-(t,x)\to 0\quad \text{as $z\to -\infty$ uniformly in $(t,x)\in\R\times\R^N$},\\
\phi(z,t,x)-p^+(t,x)\to 0\quad \text{as $z\to +\infty$ uniformly in $(t,x)\in\R\times\R^N$}.
\end{cases}
\end{equation}
\end{definition}
This definition was shown in \cite{Nadin-JMPA} to be equivalent to the one introduced by Nolen, Rudd and Xin \cite{NolenRuddXin}, where an auxiliary equation was used in the definition.
The minimal wave speed of such pulsating fronts was established in \cite{Nadin-JMPA} under a KPP type condition and the following monostability condition: (i) there is a positive continuous space-time periodic solution $p$; (ii) if $u$ is a space periodic solution such that $u\le p$ and $\inf_{(t,x)\in \R\times\R^N}u(t,x)>0$, then $u\equiv p$; (iii) $u\equiv 0$ is an unstable solution in the sense that the associated generalized eigenvalue is positive,  where the generalized eigenvalue was studied in \cite{Nadin-Ann}. An upper and a lower bounds were given for the minimal wave speed (if it exists) when the KPP condition does not hold. The spreading speed as well as the tail behavior and the regularity of the wave were also studied there. One may ask the following questions: Does the minimal wave speed exist when the KPP type condition does not hold? Can $u(t,x;y):=\phi(y+x\cdot e+ct,t,x)$ satisfy the equation for any $y\in\R$? Can such a result be established for systems admitting possible semi-trivial time-space periodic solutions?  We will give affirmative answers to these questions.

 Most recently, Rawal, Shen and Zhang \cite{RawalShenZhang14} introduced the following definition of time-space periodic traveling waves for a nonlocal dispersal Fisher-KPP equation:
\begin{definition}\label{wave-Shen}
An entire solution $u(t,x)$ is called a traveling wave solution connecting $u^*(t,x)$ and $0$ and propagating in the direction of $e$ with speed $c$ if there is a bounded function $\Phi:\R^N\times\R\times\R^N\to \R_+$ satisfying that $\Phi$ is locally Lebesgue measurable, $u(t,x;\Phi(\cdot,0,z))$ exists for all $t\in\R$,
\begin{equation*}
u(t,x;\Phi(\cdot,0,z))=\Phi(x-cte,t,z+cte),\quad t\in\R,z\in\R^N
\end{equation*}
\begin{equation*}
\lim_{x\cdot e\to -\infty}(\Phi(x,t,z)-u^*(t,x+z))=0, \quad \lim_{x\cdot e\to+\infty} \Phi(x,t,z)=0,\, \, \text{uniformly in $(t,z)$,}
\end{equation*}
\begin{equation*}
\Phi(x,t,z-x)=\Phi(x',t,z-x'),\quad \text{ $x,x'\in\R^N$ with $x\cdot e=x'\cdot e$,}
\end{equation*}
and
\begin{equation*}
\Phi(x,t+T,z)=\Phi(x,t,z+p_ie_i)=\Phi(x,t,z),\quad x,z\in\R^N.
\end{equation*}
\end{definition}
Let  $\phi(\xi,t,x):=\Phi(y,t,x-y)$ with $\xi=y\cdot e$. It then follows that for any $z\in\R^N$, $\phi(x\cdot e-ct,t,x+z)$ is a solution of the given evolution equation. Note that two quite different approaches were used in \cite{Nadin-JMPA} and \cite{RawalShenZhang14}, respectively, to prove the existence of traveling waves. For further investigations on KPP type nonlocal evolution equations, we refer to Shen \cite{ShenTAMS}, Kong and Shen \cite{KongShenJDDE}, and references therein.

This paper consists of two parts. In the first one, we establish the traveling waves and spreading speeds for an $\omega$-time
periodic and $L$-space periodic monotone semiflow $\{Q_t\}_{t\in\mc{T}}$ of monostable type on some subsets of the space $\mc{C}$ consisting of all continuous functions from one-dimensional unbounded medium $\mc{H}$ to the Banach lattice $(X,X_+,\|\cdot\|)$ with $X=C(\Omega,\R^l)$, where $\Omega$ is a compact metric space, the evolution time $\mc{T}$ is $\R_+$ or $\Z_+$, and the medium $\mc{H}$ is $\R$ or $\Z$. In the second part, we address the aforementioned questions by applying the obtained results to a time-space periodic competition model in the medium $\R$.

Below we introduce the definition of traveling wave with speed $c$ for time-space periodic semiflows.
\begin{definition}\label{wave-our-def}
A function $W: \mc{T}\times \mc{H}\times \R \to X $ is said to a traveling wave for the semiflow $\{Q_t\}_{t\in\mc{T}}$ provided that
\begin{equation}\label{wave-definition}
Q_t[W(0,\cdot,\cdot+y)](x)=W(t,x,x+y-ct), \quad \forall t\in\mc{T}, x\in\mc{H}, y\in\R,
\end{equation}
\begin{equation}\label{wave-periodicity}
\text{$W(t,x,\xi)$ is $\omega$-periodic in $t$, $L$-periodic in $x$, non-increasing in $\xi$,}
\end{equation}
and
\begin{equation}\label{wave-limit}
\text{$W(t,x,\pm \infty)$ exist such that  $Q_t[W(0,\cdot,\pm\infty)]=W(t,\cdot,\pm\infty)$.}
\end{equation}
\end{definition}
In \eqref{wave-definition}, for any $y\in\R$, $W(0,\cdot,\cdot+y)$ is understood as a one variable function which is an element of $\mc{C}$. We say that $W(t,x,\xi)$ connects $\omega$-time periodic and $L$-space periodic function $\beta_1$ to $\beta_2$ if $\lim_{\xi\to-\infty} |W(t,x,\xi)-\beta_1(t,x)|=0$ and $\lim_{\xi\to+\infty} |W(t,x,\xi)-\beta_2(t,x)|=0 $ uniformly in $t\in\mc{T}$ and $x\in\mc{H}$.

One may observe the differences between Definitions \ref{wave-Nadin} and \ref{wave-our-def}. For instance, the medium is $\R^N$ in Definition \ref{wave-Nadin} and $\R$ or $\Z$ in Definition \ref{wave-our-def}. In section 2.1 we will explain how these two definitions are relevant, why \eqref{wave-definition}  can hold for all $y\in\R$, and how the function $W$ can be extended
 to obtain an entire orbit for the given semiflow.

The mono-stability of the semiflow will be defined later in section 2.1 by using its Poincar\'e map $Q_\omega$ restricted on the space of $L$-periodic functions.
Since semi-trivial time-space periodic solutions may exist in certain non-scalar evolution systems, we
need to introduce one more critical speed. In general, there are two kinds of spreading speeds (and minimal wave speeds) for semiflows and they are not necessarily identical or linearly determinate. For specific evolution systems, it is highly nontrivial to find appropriate conditions for the existence of a single
spreading speed and its linear determinacy.
We will illustrate this by considering a two species competition model.

Our strategy to establish the minimal wave speed for the semiflow is the following: (1) construct a map $P_\omega$ in homogeneous medium $\Z$ such that it is topologically conjugate to the Poincar\'e map $Q_\omega$ in periodic medium $\mc{H}$; (2) extend $P_\omega$  to a larger map $\tilde{P}_\omega$ in homogeneous medium $\R$ but with very weak compactness even if $Q_\omega$ is compact; (3) show that $\tilde{P}_\omega$ fits the framework of \cite{FZ} to overcome the difficulty induced by non-compactness; (4) construct the wave for $\{Q_t\}_{t\in\mc{T}}$ using the evolution approach, which
is applicable because the medium of $\tilde{P}_\omega$ is $\R$. The spreading speed will be
 obtained by a similar idea but the phase space for $\tilde{P}_\omega$  is selected in a different way.

In the second part, we apply the theory developed for monotone semiflows to the following two species competition  reaction-advection-diffusion model with time and space periodicity:
\begin{eqnarray}\label{VL}
 & &\frac{\partial u_1}{\partial t}=L_1u_1+u_1(b_1(t,x)-a_{11}(t,x)u_1-a_{12}(t,x)u_2), \\
 & &\frac{\partial u_2}{\partial t}=L_2u_2+u_2(b_2(t,x)-a_{21}(t,x)u_1-a_{22}(t,x)u_2),\quad t>0, \ x\in\mathbb{R}. \nonumber
 \end{eqnarray}
Here $L_iu=d_i(t,x)\frac{\partial^2 u}{\partial x^2}-g_i(t,x)\frac{\partial u}{\partial x}, i=1,2$, $u_1$ and $u_2$
denote the population densities of two competing species in  $\omega$-time and $L$-space periodic environment, respectively, $d_i(t,x)$, $g_i(t,x)$ and $b_i(t,x)$ are  diffusion, advection and growth rates of
the $i$-th species ($i=1,2$), respectively,
and $a_{ij}(t,x)(1\le i,j\le2)$ are inter- and intra-specific competition coefficients.
In order to verify the mono-stability assumption, we first find two semi-trivial time-space periodic solutions $(u^*_1(t,x),0)$ and $(0,u^*_2(t,x))$, one of which, under a set of conditions, is shown to be globally stable for system \eqref{VL} with periodic initial datum. Since $(0,0)$ is always a solution between the two semi-trivial time-space periodic solutions with respect to the competitive ordering, there might be two spreading speeds in general. Due to the structure of competition, we can construct  upper solutions to show these two speeds (having different definitions) are identical. Some sufficient conditions for the linear determinacy of the speed are also derived. For the reaction-diffusion competition model studied in \cite{HMP} with unbounded domain, we obtain more explicit conditions  for the existence of the minimal wave speed. In the case where there is no spatial heterogeneity in \eqref{VL} (i.e., all coefficients are independent of $x$), our analysis shows that the minimal wave speed obtained in \cite{ZR} is also the single spreading speed for such a system.

For two species time-periodic and space-dependent reaction-diffusion competition models in a bounded domain, Hess and Lazer \cite{Hess2} (see also \cite{Hess})  studied the existence, stability and attractivity of nonnegative time-periodic solutions of model systems. Hutson, Mischaikow and Pol\'{a}\u{c}ik  \cite{HMP} investigated the effect of different diffusion rates on the survival of two phenotypes of a species, and showed that the interaction between temporal and spatial variability leads to a quite different result compared with the autonomous case \cite{DHMP}, which concluded that the phenotype with the slower diffusion rate always wins the competition. Meanwhile, in an unbounded domain, Zhao and Ruan \cite{ZR} obtained the existence, uniqueness and stability of time-periodic traveling waves for time-periodic but space-independent reaction-diffusion competition models. For a reaction-diffusion
competition model with seasonal succession, Ma and Zhao \cite{MZ} studied the existence of single spreading speed and its linear determinacy, and showed that the spreading
speed coincides  with the minimal wave speed of  time-periodic traveling waves. More recently, Kong, Rawal and Shen \cite{KRS} proposed a competition model with nonlocal dispersal in a time and space periodic habitats, and investigated the spreading speed and its linear determinacy.
For traveling waves in a time-delayed reaction-diffusion competition model with nonlocal terms,
we refer to Gourley and Ruan \cite{GR}. It is worthy to point out that our approach
is quite different from those in \cite{KRS,Nadin-JMPA, RawalShenZhang14}.

%

The rest of this paper is organized as follows. In the next section, we establish the theory of
traveling waves and spreading speeds for time-space periodic semiflows of monostable type.
In section 3, we apply this theory to the model system \eqref{VL} and
explore its propagation phenomena  by using the competition structure.

%
%

\section{Time-space periodic semiflows}
In this section, we first present some notations and assumptions and then study the
existence of traveling waves and spreading speeds for time-space periodic semiflows.

\subsection{Preliminaries}
Let $\Omega$ be a compact metric space, $\R^l$  be the $l$-dimensional Euclidean space and $X:=C(\Omega,\R^l)$. We endow $X$ with the maximum norm $\|\cdot\|$ and the partial ordering induced by the positive cone $X_+:=C(\Omega,\R^l_+)$. Then $(X,X_+,\|\cdot\|)$ is a Banach lattice. For $\varphi_1,\varphi_2\in X$, we write $\varphi_1\ge\varphi_2$ if $\varphi_1-\varphi_2\in X_+$, $ \varphi_1\gg \varphi_2$ if $\varphi_1-\varphi_2\in \text{Int} X_+$, and $\varphi_1>\varphi_2$ if $\varphi_1\ge \varphi_2$ but $\varphi_1\neq \varphi_2$. For $\varphi_1,\varphi_2\in X$, the least upper bound of the set $\{\varphi_1,\varphi_2\}$, denoted by $\max\{\varphi_1,\varphi_2\}$, is also an element of $X$. Moreover,
\begin{equation*}
\max\{\varphi_1,\varphi_2\}(x)=\max\{\varphi_1(x),\varphi_2(x)\}.
\end{equation*}

Let $\mc{H}=\R$ or $\Z$.  Define $[a, b]_\mc{H}$ as a closed subset of $\mc{H}$ in the sense that if $\mc{H}=\R$, then
$[a, b]_\mc{H}=[a,b]$ with $a, b\in\R$ and $a\le b$; and if $\mc{H}=\Z$, then
$[a, b]_\mc{H}=\{a,a+1,a+2,...,b\}$ with $a, b\in\Z$ and $a
\leq b$. For  $r\in \mc{H}$ with $r>0$, define $r\mathbb{Z}:=\{rh:h\in \mathbb{Z}\}$. We use $\mc{C}$ to denote all continuous and  bounded functions from $\mc{H}$ to $X$. We endow $\mc{C}$ with the compact open topology, which can be induced by the following metric
\begin{equation}\label{metric}
d(u,v):=\sum_{k=1}^\infty \f{\max_{|x|\le k}\|u(x)-v(x)\|}{2^k},\quad u,v \in \mc{C}.
\end{equation}
 A sequence $u_n$ is said to be convergent to $u$ in $\mc{C}$ provided that $u_n(x)\to u(x)$ in $X$ uniformly locally in $x\in \mc{H}$ (that is, uniformly for $x$ in any compact subset of
 $\mc{H}$). On the other hand, if $u_n\in \mc{C}$ is uniformly bounded and converges uniformly locally to some function $u$, then $u\in \mc{C}$. For $u_1,u_2\in \mc{C}$, we write $u_1\ge u_2$ if $u_1(x)\ge u_2(x)$ for all $x\in \R$. A subset $U$ of $\mc{C}$ is bounded if $\sup_{u\in U}d(u,0)$ is finite. For $u\in \mc{C}$, define the function $u_{[0,L]_{\mc{H}}}\in C([0,L]_{\mc{H}},X)$ by $u_{[0,L]_{\mc{H}}}(x)=u(x)$. Given a bounded set $U\subset \mc{C}$, we use $U_{[0,L]_{\mc{H}}}$ to denote the the set $\{u_{[0,L]_{\mc{H}}}: u\in U\}$. We use the Kuratowski measure to define the noncompactness of $U_{[0,L]_{\mc{H}}}$ which is naturally endowed with the uniform topology. The measure is defined as follows.
\begin{equation}\label{alpha-measure}
\alpha(U_{[0,L]_\mc{H}}):=\inf\{r: U_{[0,L]_\mc{H}} \, \text{has a finite open cover of diameter less than $r$}\}.
\end{equation}
The set $U_{[0,L]_\mc{H}}$ is precompact if and only if $\alpha(U_{[0,L]_\mc{H}})=0$.

Let $L\in \mc{H}$ be a positive number, We use $\mc{C}^{per}$ to denote the set of all $L$-periodic functions in $\mc{C}$. We endow $\mc{C}^{per}$ with the same topology as $\mc{C}$. But the convergence of a sequence in $\mc{C}^{per}$ will be in the following stronger sense: $u_n$ is said to be convergent to $u$ in $\mc{C}^{per}$ provided that $u_n(x)\to u(x)$ in $X$ uniformly in $x\in [0,L]_{\mc{H}}$. For $u_1,u_2\in \mc{C}^{per}$,  we write $u_1\ge u_2$ if $u_1(x)-u_2(x)\in X_+$ for all $x\in\mc{H}$, $ u_1\gg u_2$ if $u_1(x)-u_2(x)\in \text{Int} X_+$ for all $x\in\mc{H}$, and $u_1>u_2$ if $u_1\ge u_2$ but $u_1\neq u_2$.

For $x\in \mc{H}$, there exist a unique $k_{x}\in \Z$ and a unique $\theta_x\in [0,L)$ such that $x=k_x L +\theta_x$. Define $[x]_L$ by
\begin{equation}
[x]_L=k_x L.
\end{equation}
For $m\in\Z$, we have $[x+mL]_L=[x]_L+mL$.

Let $\omega\in \mc{T}$ be a positive number. Assume that $\beta: \mc{T}\times \mc{H}\to {\rm Int} X_+$ is continuous such that $\beta(t,x)$ is $\omega$-periodic in $t\in\mc{T}$ and $L$-periodic in $x\in \mc{H}$.  Then for any $t\in \mc{T}$, $\beta(t,\cdot)\in \mc{C}^{per}$ and $\beta(t,\cdot)\gg 0$. Define
\begin{equation}
\mc{C}_{\beta(t,\cdot)}:=\{\phi\in \mc{C}: 0\le \phi(x)\le \beta(t,x), x\in\mc{H}\}, \quad t\in\mc{T}
\end{equation}
and
 \begin{equation}
\mc{C}^{per}_{\beta(t,\cdot)}=\mc{C}_{\beta(t,\cdot)}\cap \mc{C}^{per}.
\end{equation}
For $y\in \mc{H}$ and any function $u:\mc{H}\to X$, define the translation operator $T_y$ by
\begin{equation}
T_y[u](x)=u(x-y).
\end{equation}
For $t\in\mc{T}$, assume that the map $Q_t:\mc{C}_{\beta(0,\cdot)}\to \mc{C}_{\beta(t,\cdot)}$ satisfies $Q_t[0]=0$ and $Q_t[\beta(0,\cdot)](x)=\beta(t,x)$.

\begin{definition}\label{def-periodic-semiflow}
$\{Q_t\}_{t\in\mc{T}}$ is an $\omega$-time periodic and $L$-space periodic monotone semiflow  from $\mc{C}_{\beta(0,\cdot)}$ to $\mc{C}_{\beta(t,\cdot)}$ provided that
\begin{itemize}
\item[(i)]$Q_0=I$, where $I$ is the identity map.
\item[(ii)]$Q_tQ_\omega =Q_{t+\omega}, \quad \forall t\in\mc{T}$.
\item[(iii)]$T_yQ_t=Q_tT_y,\quad  \forall t\in\mc{T},  y\in L\Z$.
\item[(iv)]$Q_t[\phi]$ is continuous jointly in $(t,\phi)\in \mc{T}\times \mc{C}_{\beta(0,\cdot)}$
\item[(v)]$Q_t[\phi]\ge Q_t[\psi],\, \forall t\in\mc{T}$ whenever $\phi\ge\psi$ in $\mc{C}_{\beta(0,\cdot)}$.
\end{itemize}
\end{definition}

In time-periodic dynamical systems, the period map is often called the Poincar\'e map.
We use $E$ to denote the set of all $\omega$-time periodic and $L$-space periodic solutions of the semiflow $\{Q_t\}_{t\in\mc{T}}$ from $\mc{C}_{\beta(0,\cdot)}$ to $\mc{C}_{\beta(t,\cdot)}$. Clearly, $0$ and $\beta$ are two elements of $E$. The following observation can be easily proved.

\begin{lemma}\label{pp1}
The following statements are valid:
\begin{enumerate}
\item[(i)] $p\in E$ if and only if $p(0,\cdot)$ is a fixed point of $Q_\omega: \mc{C}^{per}_{\beta(0,\cdot)}\to  \mc{C}^{per}_{\beta(0,\cdot)}$.
\item[(ii)] Let $u,v \in \mc{C}^{per}_{\beta(0,\cdot)}$. If $u$ is a fixed point of $Q_\omega$ and $\lim_{n\to\infty} Q_{n\omega}[v]=u$, then  $\lim_{t\to\infty} d(Q_t[v], Q_t[u])=0$, where $d$ is the metric defined in \eqref{metric}.
\end{enumerate}
\end{lemma}

In \eqref{wave-definition}-\eqref{wave-limit}, we have defined the traveling wave for the semiflow $\{Q_t\}_{t\in\mc{T}}$. Here we explain it in the case where the semiflow is generated by the solution maps  of a time-space periodic evolution system, including how to extend such a wave
solution to an entire solution and how it relates to the one introduced by Nadin \cite{Nadin-JMPA}.

We first explain how to extend a wave to an entire solution. For $t\ge r$ with $t,r\in \mc{T}\cup(-\mc{T})$, let  $S_{r,t}: \mc{C}_{\beta(r,\cdot)} \to \mc{C}_{\beta(t,\cdot)}$ be the solution map of a time-space periodic evolution equation in dimension one, where $r$ is the initial time. Then $S_{r,t}$ has following time periodicity:
\begin{equation}
S_{r,t}=S_{r+\omega, t+\omega}, \quad t\ge r, \quad t,r\in \mc{T}\cup(-\mc{T}).
\end{equation}
Suppose that we have already established the traveling wave $W(t,x,\xi)$ for $\{S_{0,t}\}_{t\in \mc{T}}$ in the sense of \eqref{wave-definition}-\eqref{wave-limit}. In particular,
\begin{equation}
S_{0,t}[W(0,\cdot,\cdot+y)]=W(t,\cdot,\cdot+y-ct),\quad t\in \mc{T}, y\in \R.
\end{equation}
 For convenience, we still use $W(t,x,\xi)$ to denote the periodic extension of $W$ in time. Note that for any $t\ge r$ with $t,r\in \mc{T}\cup(-\mc{T})$, there exists $k_r\in \Z_+$ such that $r+k_r\omega\geq 0$. It then follows that
\begin{eqnarray}
&&S_{r,t}[W(r,\cdot,\cdot+y-cr)]\nonumber\\
=&& S_{r+k_r\omega, t+k_r\omega}[W(r,\cdot,\cdot+y-cr)]\nonumber\\
=&& S_{r+k_r\omega, t+k_r\omega}[W(r+k_r\omega,\cdot,\cdot+y-cr)]\nonumber\\
=&& S_{r+k_r\omega, t+k_r\omega}S_{0,r+k_r\omega}[W(0,\cdot,\cdot+y-cr+c(r+k_r\omega))]\nonumber\\
=&& S_{0, t+k_r\omega}[W(0,\cdot,\cdot+y-cr+c(r+k_r\omega))]\nonumber\\
=&& W(t, \cdot,\cdot+y-ct), \quad \forall t\in \mc{T}\cup(-\mc{T}), y\in \R.
\end{eqnarray}
This shows that the periodic extension $W$ gives rise to an entire wave solution.

 Next, we point out there are many wave-like solutions satisfying \eqref{wave-definition}. Indeed, for a decreasing function $\phi\in \mc{C}_{\beta(0,\cdot)}$, we define
 \begin{equation}
 U(t,x,\xi):=Q_t[\phi(\cdot+\xi+ct-x)](x).
 \end{equation}
Then one may easily verify that $U$ is periodic in $x$, non-increasing in $\xi$ and satisfies \eqref{wave-definition}. However, $U$ is in general not periodic in time and not extendable to be an entire solution. This suggests that we first look for a wave for the Poincar\'e map (i.e., period map) in a certain sense and then use it as the initial value to evolve under the semiflow to construct the traveling wave for the semiflow. We will show that $W(t,x,\xi):=Q_t[V(\cdot,\cdot+\xi+ct-x)](x)$ is a traveling wave of the semiflow $\{Q_t\}_{t\in \mc{T}}$ if $V(x,\xi)$ is $L$-periodic in $x$, non-increasing in $\xi$ and satisfies
\begin{equation}\label{poincare-map-wave}
Q_\omega[V(\cdot,\cdot+y)]=V(\cdot,\cdot+y-c\omega),\quad \forall y\in \R.
\end{equation}
The periodicity of $W$ in time follows from \eqref{poincare-map-wave}. We call such $V$ a traveling wave of $Q_\omega$ in a strong sense.


Now let us roughly explain why \eqref{poincare-map-wave} may hold for all $y\in\R$. Indeed, we can employ the results in \cite{FZ} to show that \eqref{poincare-map-wave} holds for $y\in\R\setminus\Gamma$, where $\Gamma$ is a countable set. Since $V$ will be carefully constructed such that $V(\cdot,\xi+\cdot-[\cdot]_L)$ is left-continuous in $\xi$ and for any $\xi\in\R$ it belongs to the same compact set in periodic function spaces, one is able to use the continuity of $Q_\omega$ to pass the limit so that \eqref{poincare-map-wave} also holds for $y\in\Gamma$.  This procedure will be presented in an abstract way in section 2.2.

To establish the existence of traveling waves and spreading speeds, we need the
following two basic assumptions on time-space periodic semiflow $\{Q_t\}_{t\in\mc{T}}$:

\begin{enumerate}
\item[(A1)] {\sc (Monostability)}
\text{$\lim_{n\to\infty}Q_{n\omega} [\phi]=\beta(0,\cdot)$ for any $\phi\in \mc {C}^{per}_{\beta(0,\cdot)}$ with $\phi\gg 0$.}

\item[(A2)] {\sc ($\alpha$-contraction)}
\text{There exists $\kappa\in [0,1)$ such that}
$$\alpha((Q_\omega[U])_{[0,L]_\mc{H}})\le \kappa \alpha(U_{[0,L]_\mc{H}}), \, \, \forall U\subset \mc{C}_{\beta(0,\cdot)}$$
where $\alpha$ is the Kuratowski measure defined in \eqref{alpha-measure}.
\end{enumerate}

Suppose $u(t,x;\phi):=Q_t[\phi](x)$ solves a time-space periodic evolution equation. Since $u(t,x;\phi)$ is $L$-periodic in $x$ if $\phi$ is, we only need to consider the evolution equation with periodic initial data. By Lemma \ref{pp1}, it follows that  $\beta(t,x)$ is a time-space periodic solution of the evolution equation, and (A1) is equivalent to that the time-space periodic solution $\beta(t,x)$  attracts any solution with initial value $\phi\in \mc {C}^{per}_{\beta(0,\cdot)}$ and $\phi\gg 0$.
For a scalar reaction-diffusion equation admitting the strong maximum principle, the condition $\phi\gg 0$ may be relaxed to be $\phi>0$. For a system of reaction-diffusion equations, such a condition in general cannot be relaxed because there probably exist semi-trivial time-space periodic solutions.

Note that the assumption (A2) is for the Poincar\'e map on the phase space $\mc{C}_{\beta(0,\cdot)}$. If the Poincar\'e map $Q_\omega: \mc{C}_{\beta(0,\cdot)}\to \mc{C}_{\beta(0,\cdot)}$ is compact, then (A2) is satisfied by choosing $\kappa=0$. For a
time-space periodic evolution equation with delay, if the delay is larger than the time period $\omega$, then $Q_\omega$ is not compact but satisfies (A2).

It is worthy to point out that we do not assume that the semfilow is subhomogeneous
(or sublinear in some literature), which is often understood as the KPP type condition
for monostable semiflows.  Thus, the expected minimal wave speed may not
be linearly determinate in general.

\subsection{Traveling waves}

In this subsection, we establish the existence of traveling waves for time-space
periodic and monotone semiflows under assumptions (A1) and (A2).

Let $\{Q_t\}_{t\in\mc{T}}$ be an $\omega$-time periodic and $L$-space periodic monotone semiflow from
 $\mc{C}_{\beta(0,\cdot)}$ to $\mc{C}_{\beta(t,\cdot)}$.  We define a family of mappings $\{S_t\}_{t\in \mc{T}}$  by
\[
S_t[\phi](x)=\f{Q_t[\phi\beta(0,\cdot)](x)}{\beta(t,x)}, \quad \forall \phi \in \mc{C}_{\bf 1},
\, x\in \mathcal{H}.
\]
It easily follows that $\{S_t\}_{t\in \mc{T}}$ is an $\omega$-time periodic and $L$-space periodic monotone semiflow on $\mc{C}_{\bf 1}$. Without loss of generality, we then assume that $\beta(t,x)$ is a positive constant, denoted still by $\beta$, and hence, we may write $\mc{C}_{\beta}$ instead of $\mc{C}_{\beta(t,\cdot)}$ for any $t\ge 0$. We do not scale $L$ to be one since the habitat has been scaled to be $\Z$ if it is discrete. We do not scale $\omega$ to be one since in delay differential equations, the relationship between the time period and the delay is important.

Our strategy is to first establish a traveling wave for the Poincar\'e map in a stronger sense than usual, and then use it as an initial value for the evolution to obtain the traveling wave for the given semiflow. To establish a traveling wave for the Poincar\'e map $Q_\omega$, we first use the map $Q_\omega$ (in periodic habitat) to construct a topologically conjugate map $P_\omega$ (in homogeneous discrete habitat), and then extend $P_\omega$ into a larger map $\tilde{P}_\omega$ (in homogeneous continuous habitat), which was introduced by Weinberger \cite{Wein82} for the study of spreading speeds. In general, $\tilde{P}_\omega$ is not compact even if $Q_\omega$ is. To overcome the difficulty caused by the non-compactness, we show that $\tilde{P}_\omega$ fits the framework of \cite{FZ} which deals with a large class of monotone semiflows with weak compactness.

To explain the operators $P_\omega$ and $\tilde{P}_\omega$ in detail, we need to introduce some notations. Let  $\mc{M}$ be  the set of all  non-increasing and bounded functions from $\R$ to $Y:=C([0,L]_{\mc{H}},X)$, and
\begin{equation}
\mc{X}=\{v\in \mc{M}: v(s)(L)=v(s+L)(0), s\in\R\}.
\end{equation}
Define order intervals $X_\beta, Y_\beta$ and $\mc{X}_\beta$, respectively, by
\begin{equation}
X_\beta=[0,\beta]_{X}, \quad Y_\beta=[0,\beta]_{Y}, \quad\text{and}\quad \mc{X}_\beta=[0,\beta]_\mc{X}.
\end{equation}
Let
\begin{equation}
\bar{Y}_\beta=\{\phi\in Y_\beta:\phi(0)=\phi(L)\}.
\end{equation}
and
\begin{equation}
\mc{K}_\beta=\{\phi\in C(L\Z, Y_\beta):\phi(iL+L)(0)=\phi(iL)(L), i\in\Z\}
\end{equation}

\begin{lemma}{\sc \cite[Section 4]{FZJEMS}}
The map $F: \mc{C}_\beta\to \mc{K}_\beta$ defined by
\begin{equation}
F[\phi](iL)(\theta)=\phi(iL+\theta)
\end{equation}
is a homeomorphism. Further, the semiflow $\{P_t\}_{t\in\mc{T}}$ on $\mc{K}_\beta$ defined by
\begin{equation}
P_t=F\circ Q_t\circ F^{-1}
\end{equation}
is topologically conjugate to $\{Q_t\}_{t\in \mc{T}}$ on $\mc{C}_\beta$.
\end{lemma}
One can verify that
\begin{equation}
F^{-1}[v](x)=v(L[x])(x-L[x]), \quad v\in \mc{K}_\beta.
\end{equation}
Define the identity map $G: \mc{X}_\beta\to \mc{K}_\beta$ by
\begin{equation}
G[\phi](iL)=\phi(iL).
\end{equation}
and the $t$-parameterized map $\tilde{P}_t$ by
\begin{equation}
\tilde{P}_t[\phi](s)=P_tG[\phi(\cdot+s)](0).
\end{equation}

Next we use two lemmas to prove that $\tilde{P}_\omega$ maps $\mc{X}_\beta$ to $\mc{X}_\beta$ and that it fits the framework of \cite{FZ} in one-dimensional homogeneous continuous habitat.

\begin{lemma}\label{property-X}
The following statements on $\mc{X}$ are valid:
\begin{enumerate}
\item[(i)]$\mc{X}\neq\emptyset$. Further, any monotone function and $L$-periodic function from $\R$ to $X$ can be embedded into $\mc{X}$.
\item[(ii)] For $r\in \R$, $r\mc{X}=\mc{X}$ and $T_r\mc{X}=\mc{X}$
\item[(iii)] For $v_1,v_2\in \mc{X}$, $v_1+v_2\in \mc{X}$.
\item[(iv)] For $v_1,v_2\in \mc{X}$, the function $v$, defined by $v(x):=\max\{v_1(x),v_2(x)\}$, is also in $\mc{X}$.
\end{enumerate}
\end{lemma}
\begin{proof}
We only prove statement (i) since others are trivial.  If $f$ is an $L$-periodic or a monotone function from $\R$ to $X$, then we can define $v\in\mc{X}$ by
\begin{equation}
v(s)(\theta)=f(s+\theta).
\end{equation}
Then $v(s)$ is monotone in $s$. Further, $v(s)$ is a constant function in $s$ if $f$ is $L$-periodic.
\end{proof}

\begin{lemma}\label{property-P}
Assume that the Poincar\'e map $Q_\omega$ satisfies (A1) and (A2). Then the map $\tilde{P}_\omega:\mc{X}_\beta\to \mc{X}_\beta$ has the following properties:
\begin{itemize}
\item[(i)] $\tilde{P}_\omega\circ T_y=T_y\circ \tilde{P}_\omega,\forall y\in\R$, where $T_y$ is the $y$-length translation operator.
\item[(ii)]  $\tilde{P}_\omega:\mc{X}_\beta\to \mc{X}_\beta$ is continuous with respect to the compact open topology.
\item[(iii)]There exists $\kappa\in [0,1)$ such that for $V\in \mc{X}_\beta$, $\alpha(\tilde{P}_\omega[V](0))\le \kappa\alpha (V(0))$, where $\alpha$ is the kuratowski measure of non-compactness for bounded sets in $Y$.
\item[(iv)]$\tilde{P}_\omega[\phi]\ge \tilde{P}_\omega[\psi]$ whenever $\phi\ge\psi$ in $\mc{X}_\beta$.
\item[(v)]$\tilde{P}_\omega: \bar{Y}_\beta\to \bar{Y}_\beta$ admits exact two fixed
 points $0$ and $\beta$, and for any $\bar{\zeta}\in  \bar{Y}_\beta$ with $\bar{\zeta}\gg0$,
$\lim_{n\to\infty}\tilde{P}_{n\omega}[\bar{\zeta}]=\beta$.
\end{itemize}
\end{lemma}
\begin{proof}
We first show that $\tilde{P}_\omega$ maps $\mc{X}_\beta$ into $\mc{X}_\beta$ and then verify the five properties one by one. Let $v\in \mc{X}_\beta$ be given. By definition, we have
\begin{equation}
\tilde{P}_\omega[v](s)=FQ_\omega F^{-1}G[v(\cdot+s)](0).
\end{equation}
Then the monotonicity of $\tilde{P}_\omega[v](s)$ follows from the monotonicity of $F, Q_\omega, F^{-1}$ and $G$. Since
\begin{equation}
T_LQ_\omega= Q_\omega T_L \quad \text{and}\quad F[\phi](L)(0)=\phi(L)=F[\phi](0)(L),
\end{equation}
we obtain
\begin{equation}
FQ_\omega F^{-1}G[v(\cdot+s+L)](0)(0)=FQ_\omega F^{-1}G[v(\cdot+s)](L)(0)=FQ_\omega F^{-1}G[v(\cdot+s)](0)(L),
\end{equation}
which is equivalent to $\tilde{P}_\omega[v](s+L)(0)=\tilde{P}_\omega[v](s)(L)$.
This shows that $\tilde{P}_\omega[v]\in \mc{X}_\beta$, and hence,
$\tilde{P}_\omega$ maps $\mc{X}_\beta$ into $\mc{X}_\beta$.

Statement (i) follows directly from the definition of $\tilde{P}_\omega$. Statement (iv) follows directly from the monotonicity of $F, Q_\omega, F^{-1}$ and $G$. It remains to verify statements (ii),(iii) and (v).

To show the continuity, it suffices to prove that $\tilde{P}_\omega[v_n]\to \tilde{P}_\omega[v]$ locally uniformly if $v_n\to v$ locally uniformly. Indeed, $G[v_n(\cdot+s)](i)\to G[v(\cdot+s)](i)$ locally uniformly in $i,s$. By the topological conjugacy between $Q_\omega$ and $P_\omega$ as well as the continuity of $Q_\omega$, it follows that $P_\omega G[v_n(\cdot+s)](0) \to P_\omega G[v(\cdot+s)](0) $ locally uniformly in $s$. The continuity is then proved.

For the statement on compactness, we note that
\begin{equation}
\tilde{P}_\omega[V](0)(\theta)=Q_\omega[F^{-1}G[V]](\theta).
\end{equation}
It then follows from (A2) and definitions of $F^{-1}$ and $G$ that
\begin{equation}
\alpha(\tilde{P}_\omega[V](0))=\alpha ((Q_\omega[F^{-1}G[V]])_{[0,L]_{\mc{H}}})\le \alpha((F^{-1} G[V])_{[0,L]_{ \mc{H} }})=\alpha(V(0)).
\end{equation}

For statement (v), it follows from statement (i) that
\begin{equation}
\tilde{P}_\omega[\bar{\zeta}](s)(\theta)=\tilde{P}_\omega[\bar{\zeta}(\cdot+s)](0)(\theta)=\tilde{P}_\omega[\bar{\zeta}](0)(\theta).
\end{equation}
Further, by the definition of $\tilde{P}_\omega$, we have
\begin{equation}
\tilde{P}_\omega[\bar{\zeta}](0)(\theta)= Q_\omega F^{-1}G[\bar{\zeta}](\theta).
\end{equation}
Note that $F^{-1}G: \bar{Y}_\beta\to C^{per}_\beta$ and $F^{-1}G[\bar{\zeta}]\gg 0$ due to $\bar{\zeta}\gg0$. It then follows from (A1) that
\begin{equation*}
\tilde{P}_{n\omega}[\bar{\zeta}]=Q_{n\omega} F^{-1}G[\bar{\zeta}]\to \beta,\quad as\quad n\to\infty.
\end{equation*}
This completes the proof.
\end{proof}

With  Lemmas \ref{property-X} and \ref{property-P}, we can proceed as in \cite{FZ} to establish the existence of traveling waves for $\tilde{P}_\omega$. Indeed, let $\varpi\in \bar{Y}_\beta$ with $0\ll \varpi\ll \beta$. Choose $\phi$ to be a continuous function from $\R$ to X with the following properties: (i) $\phi(x)$ is non increasing in $x$, (ii) $\phi(x)=0$ for $x\ge0$ and (iii) $\phi(-\infty)=\varpi$.
Define $\tilde{\phi}\in \mc{X}_\beta$ by
\begin{equation}\label{phi-tilde}
\tilde{\phi}(s)(\theta)=\phi(s+\theta).
\end{equation}
Then $\tilde{\phi}$ has the following properties:
\begin{enumerate}
\item[(B1)] $\tilde{\phi}(s)$ is continuously non-increasing in $s$;
\item[(B2)] $\tilde{\phi}(s)=0$ for $s\ge0$;
\item[(B3)] $\tilde{\phi}(-\infty)=\varpi$.
\end{enumerate}
Next we use $\tilde{\phi}$ to define two numbers $-\infty<c_+^*\le \bar{c}_+\le +\infty$. For $c\in\R$ and integer $n\ge 1$, we define the map $R_{c,\f{1}{n}}: \mc{X}_\beta\to \mc{X}_\beta$ by
\begin{equation}
R_{c,\f{1}{n}}[a](s)=\max\left\{\f{1}{n} \tilde{\phi}(s), T_{-c\omega} \tilde{P}_\omega[a](s)\right\}
\end{equation}
and a sequence of functions $a_m\left(c,\f{1}{n}; s\right)$ by the recursion
\begin{equation}\label{iteration}
a_0\left(c,\f{1}{n};s\right)=\tilde{\phi}(s),\quad a_{n+1}\left(c,\f{1}{n};s\right)=R_{c,\f{1}{n}}\left[a_n\left(c,\f{1}{n};\cdot\right)\right](s).
\end{equation}
Define
\begin{equation}
A_0=\mc{X}_\beta, \quad A_{i+1}=\overline{\cup_{n\ge 1}R_{c,\f{1}{n}}[A_{i}]},\quad i\ge 1.
\end{equation}
Then we have the following result.
\begin{lemma}\label{A} {\sc \cite[Lemmas 3.1 and 3.3]{FZ}}
The following two statements hold true:
\begin{enumerate}
\item[(i)] The set $A:=\cap_{i\ge 0}\cup_{s\in\R}A_i(s)$ is non-empty and compact in $Y_\beta$.
\item[(ii)] $\lim_{m\to\infty}a_m(c,\f{1}{n};s)$ exists and the limit, denoted by $a(c,\f{1}{n};\cdot)$, is an element in $\mc{X}_\beta$ and satisfies
\begin{equation}
R_{c,\f{1}{n}}\left[a\left(c,\f{1}{n};\cdot\right)\right](s)=a\left(c,\f{1}{n};s\right),\quad a.e. \quad s\in\R
\end{equation}
and
\begin{equation}
\text{$a\left(c,\f{1}{n};-\infty\right)=\beta,\quad a\left(c,\f{1}{n};-\infty\right)\in Y_\beta$ is a fixed point of $\tilde{P}_\omega$.}
\end{equation}
\end{enumerate}
\end{lemma}

According to \cite{FZ}, we define two numbers:
\begin{equation}\label{cnumber}
c^*_+:=\sup\{c: a(c,1;+\infty)=\beta\}, \quad \bar{c}_+:=\sup\{c: a(c,1;+\infty)>0\}.
\end{equation}

Let $\tilde{E}$ be the fixed point of $\tilde{P}_\omega$ on $\bar Y_\beta$. We say $\psi(s-c\omega)$ is a traveling wave of $\tilde{P}_\omega$ connecting $\beta_1\in \tilde{E}$ to $\beta_2\in \tilde{E}$ if there exits a countable set $\Gamma\subset\R$ such that
\begin{equation}\label{wave-P-tilde-def}
\tilde{P}_{n\omega}[\psi](s)=\psi(s-cn\omega), \quad n\ge 0, s\in \R\setminus \Gamma
\end{equation}
and
\begin{equation}\label{wave-P-tilde-limit}
\psi(-\infty)=\beta_1,\quad \psi(+\infty)=\beta_2.
\end{equation}

Applying the same arguments as in the proof of  \cite[Theorem 3.8]{FZ} to the map $\tilde{P}_\omega:\mc{X}_\beta\to\mc{X}_\beta$, we have the following result.

\begin{lemma}\label{wave-P-tilde}
The following statements are valid:
\begin{enumerate}
\item[(1)] For $c\ge c^*_+$,  $\tilde{P}_\omega$ admits a traveling wave $\psi$
 connecting $\beta$ to some elements $\beta_1\in \tilde{E}\setminus \{\beta\}$.
\item[(2)] If, in addition, $0$ is isolated in $\tilde{E}$, then for any $c\ge \bar{c}_+$ either of the following holds:
    \begin{enumerate}
    \item[(i)]there exists a traveling wave $\psi$ connecting $\beta$ to $0$.
    \item[(ii)]there are two ordered elements $\beta_1,\beta_2$ in $\tilde{E}\setminus \{0,\beta\}$
    such that there exist a traveling wave $\psi_1$ connecting $\beta_1$ to $0$
    and a traveling wave $\psi_2$ connecting $\beta$ to $\beta_2$.
    \end{enumerate}
\end{enumerate}
\end{lemma}

Before moving forward to construct the traveling waves for $\{Q_t\}_{t\in\mc{T}}$ in the sense of \eqref{wave-definition}-\eqref{wave-limit}, we show that the set $\Gamma$ can be chosen to be empty for all traveling waves of $\tilde{P}_\omega$ established in Lemma \ref{wave-P-tilde}.
\begin{lemma}\label{left-c}
Let $\psi(s-c\omega)$ be a left continuous traveling wave of $\tilde{P}_\omega$ established in Lemma \ref{wave-P-tilde} in the sense of \eqref{wave-P-tilde-def}-\eqref{wave-P-tilde-limit}. Then $\tilde{P}_\omega [\psi](s)$ is also left continuous, and hence, $\tilde{P}_\omega[\psi](s)=\psi(s-c\omega),s\in\R$.
 \end{lemma}
\begin{proof}
We only prove that $\tilde{P}_\omega [\psi](s)$ is left continuous. Fix $s\in \R$, and let $s_n\uparrow 0$. We show show that $\tilde{P}_\omega [\psi](s+s_n)\to \tilde{P}_\omega [\psi](s)$ as $n\to\infty$. By the continuity of $\tilde{P}_\omega$, it suffices to show that $\psi(x+s_n)\to \psi(x)$ locally uniformly in $x\in \R$. Indeed, from the construction of $\psi$ we know that $\cup_{s\in \R}\psi(s)\subset A$, where $A$ is the compact set defined in Lemma \ref{A}. Thus, we see from \cite[Lemma 2.2]{FZ} that the discontinuous points of $\psi$ is at most countable. Thus, given a bounded set $D\subset \R$, we are able to choose a dense subset $B$ of $D$ such that $\psi(x)$ is continuous in $x\in B$. For $\epsilon>0$ and each $x\in B$, there is a neighbourhood $\mc{N}_x$ such that $\|\psi(x)-\psi(y)\|\le \epsilon/4, y\in \mc{N}_x$. Since $\{N_{x}\}_{x\in B}$ is an open cover of $D$, there exist finitely many points, say $x_i, i=1,\cdots, m$, such that $D\subset \cup_{i=1}^m \mc{N}_{x_i}$.  For any $x\in D$ and large $n$, there exists $i$ such that either $x,x+s_n\in \mc{N}_{x_i}$ or $x+s_n\in\mc{N}_{x_i}$ and $x\in\mc{N}_{x_{i+1}}$ with $\mc{N}_{x_i}\cap \mc{N}_{x_{i+1}}\neq \emptyset$. For both cases, we conclude that there exists $N>0$ such that $\|\psi(x+s_n)-\psi(x)\|\le \epsilon, x\in D, n\ge N$.
\end{proof}

 Now we are in a position to prove the main result in this subsection.

 \begin{theorem}\label{m-TW}
  Let $\{Q_t\}_{t\in\mc{T}}$ be an $\omega$-time periodic and $L$-space periodic monotone semiflow from $\mc{C}_{\beta(0,\cdot)}$ to $\mc{C}_{\beta(t,\cdot)}$, and assume that (A1) and (A2) hold. Then there are two numbers $-\infty <c^*_+\le \bar{c}_+\le +\infty$ such that
\begin{enumerate}
\item[(1)]For any $c\ge c^*_+$,  $\{Q_t\}_{t\in\mc{T}}$ admits a traveling wave $W$
 connecting $\beta$ to some elements $\beta_1\in E\setminus \{\beta\}$.
\item[(2)]If, in addition, $0$ is isolated in $E$, then for any $c\ge \bar{c}_+$ either of the following holds:
    \begin{enumerate}
    \item[(i)]there exists a traveling wave $W$ connecting $\beta$ to $0$.
    \item[(ii)]there are two ordered elements $\alpha_1,\alpha_2$ in $E\setminus \{0,\beta\}$
    such that there exist a traveling wave $W_1$ connecting $\alpha_1$ to $0$
    and a traveling wave $W_2$ connecting $\beta$ to $\alpha_2$.
    \end{enumerate}
\item[(3)]For $c<c^*_+$, there is no traveling wave connecting $\beta$, and for $c<\bar{c}_+$,
there is no traveling wave connecting $\beta$ to $0$.
\end{enumerate}
\end{theorem}

\begin{proof}
For each admissible speed $c$, we have already shown that $\tilde{P}_\omega:\mc{X}_\beta\to \mc{X}_\beta$ admits a wave $\psi$ in the sense that
\begin{equation}\label{p-tilde-TW}
\tilde{P}_\omega [\psi](s)=\psi(s-c\omega),\quad \forall s\in\R.
\end{equation}
Define $V(x,\xi)$ by
\begin{equation}
V(x,\xi)=\psi(\xi-x+[x]_L)(x-[x]_L).
\end{equation}
It is not difficult to check $V$ is $L$-periodic in $x$ and  non-increasing in $\xi$. Then we use the definitions of $F, F^{-1}, P_\omega$ and $\tilde{P}_\omega$ to obtain
\begin{eqnarray}\label{wave-Q-omega}
V(x,x-c\omega+y)&&=\psi(y-c\omega+[x]_L)(x-[x]_L)\nonumber\\
&&=\tilde{P}_\omega[\psi](y+[x]_L)(x-[x]_L)\nonumber\\
&&=P_\omega G[\psi(\cdot+y+[x]_L)](0)(x-[x]_L)\nonumber\\
&&=FQ_\omega F^{-1}G[\psi(\cdot+y+[x]_L)](0)(x-[x]_L)\nonumber\\
&&=Q_\omega F^{-1}G[\psi(\cdot+y+[x]_L)](x-[x]_L)\nonumber\\
&&=Q_\omega F^{-1}G[\psi(\cdot+y)](x)\nonumber\\
&&=Q_\omega[\psi([\cdot]_L+y)(\cdot-[\cdot]_L)](x)\nonumber\\
&&=Q_\omega [V(\cdot,\cdot+y)](x),\quad \forall x\in \mc{H},\,y\in\mathbb{R}.
\end{eqnarray}
We claim that $W:\mc{T}\times\mc{H}\times\R\to X$ defined by
\begin{equation}\label{def-W}
W(t,x,\xi)=Q_t[V(\cdot,\cdot+\xi+ct-x)](x),
\end{equation}
is the desired traveling wave. It suffices to show that \eqref{wave-definition}-\eqref{wave-limit}
hold true. Indeed, the space periodicity and \eqref{wave-definition} are easily verified. Also the limit equality \eqref{wave-limit} follows directly from the time periodicity. Thus,  it remains to prove  the time periodicity. By using \eqref{wave-Q-omega}, we obtain
\begin{eqnarray}
W(t+\omega,x,\xi)&&=Q_tQ_\omega [V(\cdot,\cdot+\xi+c(t+\omega)-x)](x)\nonumber\\
&&=Q_t[V(\cdot,\cdot-c\omega+\xi+c(t+\omega)-x)](x)\nonumber\\
&&=Q_t[V(\cdot,\cdot+\xi+ct-x)](x)\nonumber\\
&&=W(t,x,\xi), \quad t\in\mc{T}, x\in \mc{H},\xi\in\R.
\end{eqnarray}
This proves the existence of traveling waves.

Next we prove the non-existence of traveling waves. If the three-variable function $W$ is a wave in the sense of \eqref{wave-definition}-\eqref{wave-limit} and it connects $\beta$ to some other time-space periodic solution $\beta_1(t,x)$ with $\beta_1(0,\cdot)<\beta$, then we have
\begin{equation}
W(0,x,x+y-c\omega)=Q_\omega [W(0,\cdot,\cdot+y)](x),\quad x\in \mc{H}, y\in\R
\end{equation}
with
\begin{equation}
W(0,\cdot,-\infty)=\beta, \quad W(0,\cdot,+\infty)=\beta_1(0,\cdot).
\end{equation}
Recall that $\tilde{\phi}$ is defined in \eqref{phi-tilde}. Thus, we may choose $s_0>0$ such that
\begin{equation}\label{non-1}
\tilde{\phi}(s)(\theta)\le W(0,\theta,\theta+s+s_0).
\end{equation}
Note that
\begin{equation}\label{non-2}
F^{-1}G[\tilde{\phi}(\cdot+s+c\omega)](x)=\tilde{\phi}([x]+s+c\omega)(x-[x]).
\end{equation}
Combining \eqref{iteration}, \eqref{non-1} and \eqref{non-2}, we obtain
\begin{eqnarray*}
a_1(c,1, s)(\theta)&&=\max\{\tilde{\phi}(s)(\theta), T_{-c\omega}\tilde{P}_\omega[\tilde{\phi}](s)(\theta)\}\nonumber\\
&& =\max\{\tilde{\phi}(s)(\theta), FQ_\omega F^{-1}G[\tilde{\phi}(\cdot+s+c\omega)](0)(\theta)\}\nonumber\\
&& \le \max\{\tilde{\phi}(s)(\theta), FQ_\omega[W(0,\cdot,\cdot+s+s_0+cw)](0)(\theta)\}\nonumber\\
&& =\max\{\tilde{\phi}(s)(\theta), F [W(0,\cdot,\cdot+s+s_0)](0)(\theta)\}\nonumber\\
&&=W(0,\theta, \theta+s+s_0),
\end{eqnarray*}
and inductively, $a_n(c,1,s)(\theta)\le W(0,\theta, \theta+s+s_0), \, \forall n\ge 1$. This implies that
\begin{equation*}
a(c,1,+\infty)=\lim_{s\to+\infty}\lim_{n\to\infty} a_n(c,1,s) \le \lim_{s\to+\infty}\lim_{n\to\infty} W(0,\cdot, \cdot+s+s_0)=\beta_1(0,\cdot)<\beta,
\end{equation*}
which, together with the definition of $c^*_+$, implies that $c\ge c^*_+$. Similarly, if $W$ is a traveling wave connecting $\beta$ to $0$ with speed $c$, then $c\ge \bar{c}_+$.
\end{proof}

We remark that there are examples arising from nonlocal or fractional diffusion equations such that $c_+^*=+\infty$. It is an interesting problem to find conditions to exclude such a possibility for semiflows, but it beyonds the purpose of this paper. When $c_+^*=+\infty$ ( or $\bar{c}_+=+\infty$), the condition $c\ge c_+^*$  (or $c\ge \bar{c}_+$) in Theorem \ref{m-TW} is
vacuous, and hence, there are no traveling waves.

 \subsection{Spreading speeds}
So far, we have already proved that there exist two critical speeds $c^*_+$ and $\bar{c}_+$ for non-increasing traveling waves. In this subsection, we use these two numbers to describe the rightward spreading property of solutions with appropriate initial datum.

\begin{theorem}\label{m-SS}
 Let $\{Q_t\}_{t\in\mc{T}}$ be an $\omega$-time periodic and $L$-space periodic monotone semiflow from  $\mc{C}_{\beta(0,\cdot)}$ to $\mc{C}_{\beta(t,\cdot)}$,, and assume that (A1) and (A2) hold. Then the following statements are valid:
\begin{enumerate}
\item[(i)]For $c>\bar{c}_+$, we have $\lim_{x\ge ct}Q_t[\phi](x)=0$ provided that $\phi\in \mc{C}_{\beta(0,\cdot)}$ vanishes when $x$ is greater than some $x_0\in \mc{H}$ and $\phi\le \varpi\ll \beta(0,\cdot)$ for some $\varpi \in \mc{C}^{per}_{\beta(0,\cdot)}$ with $\varpi\gg 0$.

\item[(ii)] For $c<c^*_+$, we have $\lim_{t\to\infty,x\le ct}|Q_t[\phi](x)-\beta(t,x)|=0$ provided that $\phi\in \mc{C}_{\beta(0,\cdot)}$ satisfies $\phi(x)\ge \sigma$ when $x$ is less than some $K\in\mc{H}$ for some $\sigma\in X$ with $\sigma \gg 0$.
\end{enumerate}
\end{theorem}

\begin{proof}
To prove these spreading properties, we again use $\tilde{P}_\omega$ but on another phase space $\mc{Y}_\beta$, which is defined by
\begin{equation}
\mc{Y}_\beta= \{v\in C(\R, Y_\beta): v(s)(L)=v(s+L)(0)\},
\end{equation}
equipped with the compact open topology.
It was shown in \cite[Proposition 6.1]{YZ} that $\tilde{P}_\omega$ maps $\mc{Y}_\beta$ to $\mc{Y}_\beta$ and $\tilde{P}_\omega$ admits the five properties in Lemma \ref{property-P} with $\mc{X}_\beta$ replaced by $\mc{Y}_\beta$. Note that different notations are used in \cite[Proposition 6.1]{YZ}.  Thus, $\tilde{P}_\omega:\mc{Y}_\beta \to \mc{Y}_\beta$ has the same spreading property as in \cite[Remark 3.2]{FZ}. On the other hand, since the semiflow $\{\tilde{P}_t\}_{t\in\mc{T}}$ is time periodic and defined in the medium $\R$, one may see from \cite[Theorem 2.3]{LYZ} that
$\{\tilde{P}_t\}_{t\in\mc{T}}$ has the spreading properties stated in Theorem \ref{m-SS} with $\mc{C}_{\beta(0,\cdot)}, Q_t,$ and $X$ replaced by $\mc{Y}_\beta, \tilde{P}_t$ and $Y$ respectively. Next we show how to derive the spreading property of $\{Q_t\}_{t\in\mc{T}}$ from $\tilde{P}_\omega$.

We look for the spreading properties of $Q_t$ which are inherited from $\tilde{P}_t$. Indeed, for $\phi\in C(\mc{H}, X)$, define $v\in \mc{Y}_\beta$ by
\begin{equation}
v(s)(\theta)=\phi\left([s]_L+\theta\right).
\end{equation}
Then  for any $s\in\R, \theta\in[0,L]$ and $n\ge 1$, we have
\begin{eqnarray}
\tilde{P}_{t}[v](s)(\theta)&&=P_{t}G[v(\cdot+s)](0)(\theta)\nonumber\\
&&=FQ_{t}F^{-1}G[v(\cdot+s)](0)(\theta)\nonumber\\
&&=Q_{t}F^{-1}G[v(\cdot+s)](\theta)\nonumber\\
&&= Q_{t}[v([\cdot]_L+s)(\cdot-[\cdot]_L)](\theta).
\end{eqnarray}
In particular, setting $s=iL, i\in Z$, we obtain
\begin{equation}
\tilde{P}_{t}[v](iL)(\theta)=Q_{t}[v([\cdot]_L)(\cdot-[\cdot]_L)](\theta+iL)=Q_{t}[\phi](\theta+iL),
\end{equation}
which is equivalent to
\begin{equation}
Q_{t}[\phi](x)=\tilde{P}_{t}[v]([x]_L)(x-[x]_L),\quad x\in\mc{H},t\in\mc{T}.
\end{equation}
Thus, the statements for $\{Q_t\}_{t\in\mc{T}}$ hold true.
\end{proof}

To finish this section, we note that similar arguments can be used to establish
the existence of non-decreasing traveling waves $W(t,x,x+ct)$ and the leftward
spreading property in terms of two critical speeds $c^*_-$ and $\bar{c}_-$ satisfying
$-\infty <c^*_-\le \bar{c}_-\le +\infty$.

\section{Two species competition model}

In this section, we  first use the abstract results in last section to study the propagation
propagation phenomena for a two species competition reaction-advection-diffusion system in time-space periodic environment. Then we obtain sufficient conditions for these two speeds to be identical
and linearly determinate, respectively. Two specific
cases are also discussed in detail.

\subsection{The periodic initial value problem}

In this subsection, we investigate the global dynamics of the time and space periodic Lotka-Volterra competition
system with the periodic initial values.
Let $\omega$ and $L$ be positive real numbers. We assume that
 \begin{enumerate}
 	\item[(a)]
 $d_i(t,x)$, $g_i(t,x)$, $a_{ij}(t,x)$ and $b_i(t,x)$ are $\omega$-periodic in $t$ and $L$-periodic in $x$, $d_i, g_i,  a_{ij}, b_i\in C^{\frac{\nu}{2},\nu}(\mathbb{R}\times\mathbb{R})$, $1\le i, j\le2,$ where $ C^{\frac{\nu}{2},\nu}(\mathbb{R}\times\mathbb{R})$ is a H\"{o}lder continuous space with the H\"{o}lder exponent $\frac{\nu}{2}$ for the first component and $\nu\in(0,1)$ for the second component.
 \item[(b)]$a_{ij}(t,x)>0$, $\forall (t,x)\in\mathbb{R}\times\mathbb{R}$, $1\le i, j\le2$.
 \item[(c)]There exists a positive number $\alpha_0$ such that $d_i(t,x)\ge\alpha_0, \forall (t,x)\in\mathbb{R}\times\mathbb{R},i=1,2$, i.e., the operator $L_iu=d_i(t,x)\frac{\partial^2 u}{\partial x^2}-g_i(t,x)\frac{\partial u}{\partial x}$ is uniformly elliptic.
\end{enumerate}
In the sequel,  if there is no specific mention, the periodicity will always refer to the time and space periods $(\omega,L)$.

Let $\mc{P}$ be the set of all continuous and $L$-periodic functions from $\mathbb{R}$ to $\mathbb{R}$ equipped with the maximum norm $\|\cdot\|_\mc{P}$, and $\mc{P}_+=\{\psi\in \mc{P}: \ \psi(x)\ge0,\forall x\in\mathbb{R}\}$ be a positive cone of $\mc{P}$. Then $(\mc{P},\mc{P}_+)$ is a strongly ordered Banach lattice.
Assume that time-space periodic functions $d, g, h\in C^{\frac{\nu}{2},\nu}(\mathbb{R}\times\mathbb{R})$ and $d(\cdot,\cdot)>0$.
It then follows that the scalar parabolic  eigenvalue problem
\begin{eqnarray}\label{VLpep1}
& & -\frac{\partial v}{\partial t}+d(t,x)\frac{\partial^2 v}{\partial x^2}-g(t,x)\frac{\partial v}{\partial x}+h(t,x)v=\lambda v, \quad (t, x)\in\mathbb{R}\times\mathbb{R},\nonumber\\
& & v(t,x+L)=v(t,x),\ v(t+\omega,x)=v(t,x),\quad \forall(t,x)\in\mathbb{R}\times\mathbb{R}
\end{eqnarray}
admits a principal eigenvalue $\lambda(d,g,h)$ associated with a positive time-space periodic eigenfunction $\phi(t,x)$(see, e.g., \cite{Hess}).
Using the arguments similar to those in \cite[Theorem 2.3.4]{Zhaobook}, as applied to the Poincar\'{e} map associated with system \eqref{VLseq},  we have the following result.
\begin{proposition}\label{VLexistence}Assume that time-space periodic functions $d, g,c, e\in C^{\frac{\nu}{2},\nu}(\mathbb{R}\times\mathbb{R})$, and $d(\cdot,\cdot)>0, e(\cdot,\cdot)\ge0$ ($\not\equiv 0$).
Let $u(t,x,\phi)$ be the unique solution of the following parabolic equation:
\begin{eqnarray}\label{VLseq}
& & \frac{\partial u}{\partial t}=d(t,x)\frac{\partial^2 u}{\partial x^2}-g(t,x)\frac{\partial u}{\partial x}+u(c(t,x)-e(t,x)u), \quad t>0,\  x\in\mathbb{R},\nonumber\\
& & u(0,x)=\phi(x)\in \mc{P}_+, \quad x\in \mathbb{R}.
\end{eqnarray}  Then the following statements are valid:
\begin{enumerate}
\item[(i)] If $\lambda(d,g,c)\le0$, then $u=0$ is globally asymptotically stable with
respect to initial values in $\mc{P}_+$;
\item[(ii)] If $\lambda(d,g,c)>0$, then \eqref{VLseq} admits a unique positive time-space periodic solution
$u^*(t,x)$, and it is globally asymptotically stable with respect to initial values
in $\mc{P}_+\backslash\{0\}$.
\end{enumerate}
\end{proposition}
Let $\mathbb{P}=PC(\mathbb{R},\mathbb{R}^2)$ be the set of all continuous and $L$-periodic functions from $\mathbb{R}$ to $\mathbb{R}^2$, and $\mathbb{P}_+=\{\psi\in \mathbb{P}: \ \psi(x)\ge0,\forall x\in\mathbb{R}\}$. Then $\mathbb{P}_+$ is a closed cone of $\mathbb{P}$ and induces a partial ordering on $\mathbb{P}$. Moreover, we introduce a norm $\|\phi\|_\mathbb{P}$ by
\[\|\phi\|_\mathbb{P}=\max_{x\in \mathbb{R}}|\phi(x)|.\]
It then follows that $(\mathbb{P},\|\cdot\|_\mathbb{P})$ is a Banach lattice.

Clearly, for any $\varphi\in\mathbb{P}_+$,  \eqref{VL} has a unique nonnegative solution $u(t,\cdot,\varphi)$ defined on $[0,\infty)$, and $u(t,\cdot,\varphi)\in\mathbb{P}_+$ for all $t\ge0$.

By Proposition \ref{VLexistence}, we see that there exist two positive time-space periodic functions $u^*_1(t,x)$ and $u^*_2(t,x)$ such that $E_1:=(u^*_1(t,x),0)$, $E_2:=(0,u^*_2(t,x))$ are the time-space periodic solutions of system \eqref{VL} provided that $\lambda(d_i,g_i,b_i)>0,\ i=1,2.$
Since we mainly concern about the case of the competition exclusion, we impose the following conditions on  system \eqref{VL}:
\begin{enumerate}
\item[(H1)]$\lambda(d_i,g_i,b_i)>0,\ i=1,2.$
\item[(H2)]$\lambda(d_1,g_1,b_1\!-\!a_{12}u^*_2)>0.$
\item[(H3)]System \eqref{VL} has no positive time-space periodic solution.
\end{enumerate}

Condition (H1) guarantees the existence of two semi-trivial time-space periodic solutions of system \eqref{VL}, and (H2) implies that $(0,u^*_2(t,x))$ is unstable.
Moreover, by Lemma \ref{mp} with $\mu=0, d(t,x)=d_1(t,x)$ and $g(t,x)=g_1(t,x),\forall (t,x)\in\mathbb{R}\times\mathbb{R}$, we know that (H2) implies $\lambda(d_1,g_1,b_1)>0$. Thus, we could simply drop the assumption $\lambda_1(d_1,g_1,b_1)>0$ from (H1).

Slightly modifying the arguments in \cite[Example 34.2]{Hess}, we have the following observation.
\begin{proposition}\label{c}
Let $\underline{b}_{1}(t):=\min\limits_{x\in[0,L]}b_{1}(t,x)$, $\overline{b}_{2}(t):=\max\limits_{x\in[0,L]}b_{2}(t,x)$, and
 $\overline{a}_{11}(t)$, $\overline{a}_{12}(t)$, $\underline{a}_{21}(t)$, $\underline{a}_{22}(t)$
 be defined in a similar way. Then (H3) holds true provided that
$$
\int^T_0\underline{b}_1(t)dt>\displaystyle\max\limits_{t\in[0,\omega]}
\frac{\overline{a}_{12}(t)}{\underline{a}_{22}(t)}\cdot\int^T_0\overline{b}_2(t)dt, \, \text{and}\, \, \int^T_0\overline{b}_2(t)dt\le\displaystyle\max\limits_{t\in[0,\omega]}
\frac{\underline{a}_{21}(t)}{\overline{a}_{11}(t)}\cdot\int^T_0\underline{b}_1(t)dt.
$$
\end{proposition}
	
Under assumptions (H1)--(H3), there are three nonnegative time-space periodic solutions: $E_0=(0,0)$, $E_1:=(u^*_1(t,x),0)$, and $E_2:=(0,u^*_2(t,x))$. Next, we use the theory developed in \cite{Hsu} for abstract competitive systems (see
also \cite{Hess2}) to prove the global stability of $E_1$.
\begin{theorem}\label{VLEQ}
Assume that (H1)--(H3) hold. Then $E_1 (u^*_1(t,x), 0)$ is globally asymptotically stable for all
initial values $\phi=(\phi_1,\phi_2)\in\mathbb{P}_+$ with $\phi_1\not\equiv0$.
\end{theorem}
\begin{proof}
Since (H2) holds true, the arguments similar to those in \cite[Proposition 7.1.1]{Zhaobook} imply the following observation.
\smallskip

\noindent {\it Claim.} There exists $\delta_0>0$ such that  $\limsup_{n\rightarrow\infty}\|u(n\omega,x,\phi)-(0,u^*_2(0,x))\|_\mathbb{P}\ge\delta_0$
for any $\phi\in\mathbb{P}_+$ with $\phi_1\not\equiv0$.

   By the above claim and (H3), we rule out possibility (a) and (c) in \cite[Theorem A]{Hsu} with $T(\phi)=u(\omega,\cdot,\phi)$. Since $E_2$ is repellent in some neighborhood of itself, \cite[Theorem A]{Hsu} implies that $E_1$ is globally asymptotically stable
   for all initial values  $\phi\in\mathbb{P}_+$ with $\phi_1\not\equiv0$.
\end{proof}
\subsection{Spreading speeds and traveling waves}
In this subsection, we study the spreading speeds and time-space periodic traveling waves for system \eqref{VL}.
By a change of variables $v_1=u_1, v_2=u^*_2(t,x)-u_2$, we transform system \eqref{VL} into the following cooperative system:
\begin{align}\label{NModel}
&\frac{\partial v_1}{\partial t}\!=\!L_1v_1\!+\!v_1(b_1(t,x)\!-\!a_{12}(t,x)u^*_2(t,x)\!-\!a_{11}(t,x)v_1\!+\!a_{12}(t,x)v_2),\quad t>0,\  x\in\mathbb{R},\nonumber \\
&\frac{\partial v_2}{\partial t}\!=\!L_2v_2\!+\!a_{21}(t,x)v_1(u^*_2(t,x)\!-\!v_2)
+\!v_2(b_2(t,x)\!-\!2a_{22}(t,x)u^*_2(t,x)\!+\!a_{22}(t,x)v_2).
\end{align}
Note that three time-space solutions of \eqref{VL}, respectively, become
\[\hat E_0=(0,u^*_2(t,x)),\ \hat E_1=(u^*_1(t,x),u^*_2(t,x)),\ \hat E_2=(0,0).\]

To apply Theorems \ref{m-TW} and \ref{m-SS} to \eqref{NModel}, we need to specify the meaning of the notations there. More precisely,
\begin{equation}
X=\R^2,\quad  \mc{H}=\R, \quad \mc{T}=\R_+,\quad\beta(t,x)=\hat E_1,\quad 0=\hat E_2.
\end{equation}
Other notations such as $\mc{C}_{\beta(0,\cdot)}$ and $\mc{C}^{per}_{\beta(0,\cdot)}$ in Theorems \ref{m-TW} and \ref{m-SS} are then accordingly specified.

 Let $\mathbb{Y}$ be the set of all bounded and continuous functions
from $\mathbb{R}$ to $\mathbb{R}$, $T_1(t,s)$ and $T_2(t,s)$ be the linear semigroups on  $\mathbb{Y}$ generated by $$\frac{\partial v}{\partial t}\!=\!L_1v+v(b_1(t,x)-a_{12}(t,x)u^*_2(t,x))$$ and $$\frac{\partial v}{\partial t}\!=\!L_2v
    +v(b_2(t,x)-2a_{22}(t,x)u^*_2(t,x)),$$ respectively. It follows that $T_1(t,s)$ and $T_2(t,s)$ are compact with the respect to the compact open topology for each $t>s\ge0$ (see, e.g., \cite{Hess}). For any $u=(u_1,u_2)\in \mathcal{C}$, define $F:[0,+\infty)\times\mathcal{C}\rightarrow \mathcal{C}$ by
    $$F(t,u)=\left(\begin{array}{c}
  -a_{11}(t,\cdot)u^2_1+a_{12}(t,\cdot)u_1u_2\\
  a_{21}(t,\cdot)u^*_2(t,\cdot)u_1-a_{21}(t,\cdot)u_1u_2+a_{22}(t,\cdot)u^2_2
    \end{array}\right).$$ Then we rewrite system \eqref{NModel} as an integral equation form:
    \begin{align}\label{inteq}& v(t)=T(t,0)v(0)+\int_{0}^{t}T(t,s)F(s,v(s))ds,\quad t>0,\nonumber\\
    & v(0)=\phi\in\mathcal{C}_{\beta(0,\cdot)},
    \end{align}
    where $T(t,s)=diag(T_1(t,s),T_2(t,s))$.

    As usual, a solution of (\ref{inteq}) is called a mild solution of system \eqref{NModel}.
    It then follows that for any $\phi\in \mathcal{C}_{\beta(0,\cdot)}$, system \eqref{NModel} has a mild solution $u(t,\cdot,\phi)$ defined on $[0,\infty)$ with $u(0,\cdot,\phi)=\phi$, and $u(t,\cdot,\phi)\in\mathcal{C}_{\beta(t,\cdot)}$ for all $t\ge0$, and it is a classical solution when $t>0$.        	 \begin{definition}
    	A function $u(t):=u(t,\cdot)$ is said to be an upper (a lower) solution of system \eqref{NModel} if it satisfies
    	$$ 
    	u(t)\geq(\leq) T(t,0) u(0)+ \int_0^t T(t,s)F(s,u(s))ds,\quad t\ge0.
    	$$
    	 \end{definition}
    Define a family of operators  $\{Q_t\}_{t\ge0}$ from $\mathcal{C}_{\beta(0,\cdot)}$ to $\mathcal{C}_{\beta(t,\cdot)}$ by  $Q_t(\phi):=u(t,\cdot,\phi)$, where $u(t,\cdot,\phi)$ is the solution of system \eqref{NModel} with $u(0,\cdot)=\phi\in\mathcal{C}_{\beta(0,\cdot)}$. Next we show that $\{Q_t\}_{t\ge0}$ is an $\omega$-time periodic and $L$-space periodic monotone semiflow from $\mathcal{C}_{\beta(0,\cdot)}$ to $\mathcal{C}_{\beta(t,\cdot)}$
     in the sense of Definition \ref{def-periodic-semiflow}. Indeed,  since for any $a\in L\mathbb{Z}$, $v(t,x):=u(t,x-a,\phi)$ and $w(t,x):=u(t+\omega,x,\phi)$ are  solutions of \eqref{NModel} with initial conditions $v(0,x)=u(0,x-a,\phi)$ and $w(0,x)=u(\omega,x,\phi)$, respectively,  we see that $Q_t$ satisfies the second and third properties in Definition \ref{def-periodic-semiflow}. The fourth property and (A2) follow from the same argument as in  \cite[Theorem 8.5.2]{Martin}.  The fifth property is true since system \eqref{NModel} is cooperative and the comparison principle holds. Moreover, Theorem \ref{VLEQ} implies (A1) is valid. Thus,  we have the
    following observation.
    \begin{proposition}\label{verification}
    Assume that (H1)--(H3) hold. Then $\{Q_t\}_{t\ge0}$ is an $\omega$-time periodic and $L$-space periodic monotone semiflow from $\mathcal{C}_{\beta(0,\cdot)}$ to $\mathcal{C}_{\beta(t,\cdot)}$, and $Q_\omega$ satisfies (A1) and (A2).
    \end{proposition}
By Proposition \ref{verification}, we see that $\{Q_t\}_{t\ge0}$ satisfies all conditions in Theorem \ref{m-TW}. Thus, there exist two numbers $\bar{c}_+$ and $c^*_+$ for the minimal speed of different kind of traveling waves. In Theorem \ref{m-TW}, $\bar{c}_+$ might be plus infinite and the information of the limits of of traveling waves at $\pm\infty$ is not fully understood for general semiflows. Below, we will use the structure of competition to show that $\bar{c}_+$ is finite and derive some conditions under which the limits of of traveling waves at $\pm\infty$ can be figured out. By the comparison
arguments, it is easy to see that $\bar{c}_+\le \max\{c^*_{1+},c^*_{2+}$\}, where $c^*_{i+}$ is the rightward spreading speed of the $u_i$ species in the absence of  the $u_{3-i}$ species, $i=1,2$. Since $c^*_{1+}$ and $c^*_{2+}$ are determined by two Fisher-KPP type equations (see \eqref{u1} and \eqref{u2} below), it follows that $\bar{c}_+<+\infty$.

 To show that $\overline{c}_+$ is the minimal wave speed for periodic traveling waves of system \eqref{NModel} connecting $\beta(t,x)$ to $0$, we propose the following assumption:
\begin{enumerate}
\item[(H4)]$c^*_{1+}+c^*_{2-}>0$, where $c^*_{1+}$ and $c^*_{2-}$ are the rightward and leftward spreading speeds
of  two Fisher-KPP type equations \eqref{u1} and \eqref{u2}, respectively.
\end{enumerate}

Note that $c^*_{1+}$ is the rightward spreading speed of $u_1$ species when $u_2$ species vanishes, and $c^*_{2-}$ is the leftward  spreading speed of $u_2$ species when $u_1$ species vanishes. When  two species have opposite advection, they may separate even without competition. Assumption (H4) excludes such a possibility so that the competition plays a vital role. We remark that if $L_iu=\frac{\partial}{\partial x}(d_i(x)\frac{\partial u}{\partial x})$ with $d_i\in C^{\frac{v}{2},1+\nu}(\mathbb{R}\times\mathbb{R})$, or all the coefficient functions in \eqref{u1} and \eqref{u2} are  even in $x$ except $g_i$ is odd in $x$, $i=1,2$, Lemma \ref{H45} shows that (H1) and (H2) guarantee (H4).

\begin{theorem}\label{MIN}
Assume that (H1)--(H4) hold. Then for any $c\ge\bar{c}_+$, system \eqref{NModel} admits a periodic traveling wave $(U(t,x, x-ct),V(t,x,x-ct))$ connecting $\beta(t,x)$ to $0$, with wave profile components $U(t,x,\xi)$ and $V(t,x,\xi)$ being continuous and non-increasing in $\xi$, and for any $c<\bar{c}_+$, there is no such traveling wave connecting $\beta(t,x)$ to $0$.
\end{theorem}
\begin{proof}
In view of Theorem \ref{m-TW} (2) and (3), it suffices to rule out the second case in Theorem \ref{m-TW} (2). Suppose, by contradiction, that the statement in Theorem \ref{m-TW} (2)(ii) is valid for some $c\ge\bar{c}_+$. Note that system \eqref{NModel} has exactly three time-space periodic solutions and $\hat E_0=(0,u_2^*(t,x))$ is the only intermediate time-space periodic solution between $\hat E_1=\beta(t,x)$ and $\hat E_2=0$, then we have $\alpha_1=\alpha_2=\hat E_0$. Thus, by restricting  system \eqref{NModel} on the order interval $[\hat{E}_0,\hat{E}_1]$ and $[\hat{E}_2,\hat{E}_0]$, respectively, we see that one scalar equation
\begin{equation}\label{u1}
 u_t=L_1u+u(b_1(t,x)-a_{11}(t,x)u)
\end{equation}
admits a periodic traveling wave $U(t,x,x-ct)$ connecting $u^*_1(t,x)$ to $0$ with $U(t,x,\xi)$ being continuous and nonincreasing in $\xi$, and the other scalar equation
\begin{equation}
v_t=L_2v+v(b_2(t,x)-2a_{22}(t,x)u^*_2(t,x)+a_{22}(t,x)v)
\end{equation}
also admits a periodic traveling wave $V(t,x,x-ct)$ connecting $u_2^*(t,x)$
to $0$ with $V(t,x,\xi)$ being continuous and nonincreasing in $\xi$.

Let $W(t,x,x-ct)=u^*_2(t,x)-V(t,x,x-ct)$. Then $W(t,x,x-ct)$ is a periodic traveling wave connecting $0$ to $u^*_2(t,x)$  of the following scalar equation with $W(t,x,\xi)$ being continuous and nondecreasing in $\xi$
\begin{equation}\label{u2}
w_t=L_2w+w(b_2(t,x)-a_{22}(t,x)w).
\end{equation}
  Note that $W(t,x,x-ct)$ is a periodic leftward traveling wave connecting $0$ to $u^*_2(t,x)$ with wave speed $-c$, and that systems \eqref{u1} and \eqref{u2} admit rightward spreading speed $c^*_{1+}$ and leftward spreading speed $c^*_{2-}$, respectively, which are also the rightward and the leftward minimal wave speeds (see, e.g., \cite[Theorem 2.1--2.3]{LYZ}). It then follows that $c\ge c^*_{1+}$ and  $-c\ge c^*_{2-}$. This implies that $c^*_{1+}+c^*_{2-}\le 0$, a contradiction.
\end{proof}
Let $\lambda_2(\mu)$ be the principle eigenvalue of the parabolic eigenvalue problem:
\begin{align}\label{eep2}
&\lambda \psi=-\frac{\partial\psi}{\partial t}+d_2(t,x)\frac{\partial^2\psi}{\partial x^2}-\!(2\mu d_2(t,x)+g_2(t,x))\frac{\partial\psi}{\partial x}\nonumber\\
&+\left(d_2(t,x)\mu^2\!+\!g_2(t,x)\mu\!+\!b_2(t,x)\!-\!a_{22}(t,x)u^*_2(t,x)\right)\psi,\quad (t,x)\in\mathbb{R}\times\mathbb{R},\\
&\psi(t,x+L)=\psi(t,x),\ \psi(t+\omega,x)=\psi(t,x),\quad (t,x)\in\mathbb{R}\times\mathbb{R}.\nonumber\end{align}

In order to prove that system \eqref{NModel} admits a single rightward spreading speed, we impose the following assumption:
\begin{enumerate}

\item[(H5)] $\limsup_{\mu\to 0^+}\frac{\lambda _2(\mu)}{\mu}\leq c_{1+}^*$, where $c_{1+}^*$ is the rightward spreading speed of \eqref{u1}.
\end{enumerate}

 By virtue of Lemma \ref{H45}, it follows that in the case where $L_iu=\frac{\partial}{\partial x}(d_i(x)\frac{\partial u}{\partial x})$ with $d_i\in C^{1+\nu}(\mathbb{R})$, or  all the coefficient functions of system \eqref{NModel} are even in $x$ except $g_i$ is odd in $x$, $i=1,2$, (H5) is automatically satisfied provided that (H1) and (H2) hold true.
\begin{theorem}\label{Qspreading}
Assume that (H1)--(H5) hold. Then the following statements are valid for system \eqref{NModel}:
\begin{enumerate}
		\item[(i)]If $\phi\in\mathcal{C}_{\beta(0,\cdot)}$, $0\le \phi\le \varpi \ll \beta(0,\cdot)$ for some $\varpi\in \mathcal{C}^{per}_{\beta(0,\cdot)}$, and $\phi(x)=0, \forall x\ge H$, for some $H\in \mathbb{R}$, then $\lim_{t\rightarrow\infty,x\ge ct}u(t,x,\phi)=0$ for any $c>\bar{c}_+$.
		\item[(ii)]If $\phi\in\mathcal{C}_{\beta(0,\cdot)}$ and $\phi(x)\ge \sigma$, $\forall x\le K$, for some $\sigma\in \mathbb{R}^2$ with $\sigma\gg0$ and $K\in\mathbb{R}$, then $\lim_{t\rightarrow\infty,x\le ct}(u(t,x,\phi)-\beta(t,x))=0$ for any $c<\bar{c}_+$.
	\end{enumerate}
\end{theorem}
\begin{proof}
In view of Theorem \ref{m-SS}, it suffices to show that $\overline{c}_+=c^*_+$. If this is not valid, then the definition of $\bar{c}_+$ and $c^*_+$ implies that $\bar{c}_+>c^*_+$. By Theorem \ref{m-TW} (1) and (3), it follows that system \eqref{NModel} admits a periodic traveling wave $(U_1(t,x,x-c^*_+t),U_2(t,x,x-c^*_+t))$ connecting $(u^*_1(t,x),u^*_2(t,x))$ to $(0,u^*_2(t,x))$ with $U_i(t,x,\xi) (i=1,2)$ being continuous and nonincreasing in $\xi$. Therefore, $U_2\equiv u^*_2(t,x)$, and $U_1(t,x,x-c^*_+t)$ is a periodic traveling wave connecting $u^*_1(t,x)$ to $0$. This implies $c^*_+\ge c^*_{1+}$ where $c^*_{1+}$ is the rightward spreading of \eqref{u1}. By \cite[Theorem B]{LYZ} (also see \cite[Theorem 3.1]{Liang}), it follows that $c^*_{1+}=\inf_{\mu>0}\frac{\lambda_1(\mu)}{\mu}$, where $\lambda_1(\mu)$ is the principal eigenvalue of the scalar parabolic eigenvalue problem:
\begin{eqnarray}\label{eep}
& & \lambda \psi=-\frac{\partial\psi}{\partial t}+d_1(t,x)\frac{\partial^2\psi}{\partial x^2}-(2\mu d_1(t,x)+g_1(t,x))\frac{\partial\psi}{\partial x}\nonumber\\
& &+(d_1(t,x)\mu^2+g_1(t,x)\mu+b_1(t,x))\psi,\quad (t,x)\in\mathbb{R}\times\mathbb{R},\\
& &\psi(t,x+L)=\psi(t,x),\ \psi(t+\omega,x)=\psi(t,x),\quad (t,x)\in\mathbb{R}\times\mathbb{R}.\nonumber
\end{eqnarray}
 For any given $c_1\in(c^*_+,\bar{c}_+)$, there exists $\mu_1>0$ such that $c_1=\frac{\lambda_1(\mu_1)}{\mu_1}$. Let $\phi^*_1(t,x)$ be the positive time and space periodic
eigenfunction associated with the principal eigenvalue $\lambda_1(\mu_1)$ of \eqref{eep}. Then it easily follows that
$$u_1(t,x):=e^{-\mu_1(x-c_1t)}\phi^*_1(t,x)=e^{-\mu_1x}e^{\lambda_1(\mu_1)t}\phi^*_1(t,x),\quad t\ge0,\ x\in\mathbb{R},$$
is a solution of the linear equation
$$
\frac{\partial u_1}{\partial t}=L_1u_1+b_1(t,x)u_1.
$$
Since $c^*_{1+}<c_1$ and (H5) holds,  we can choose a
small number $\mu_2\in (0, \mu_1)$ such that  $c_2:=\frac{\lambda_2(\mu_2)}{\mu_2}< c_1$.
  Let $\phi_2^*(t,x)$ be the positive time and space periodic
eigenfunction associated with the principal eigenvalue $\lambda_2(\mu_2)$ of \eqref{eep2}.
It is easy to see that
$$
u_2(t,x):=e^{-\mu_2(x-c_2t)}\phi^*_2(x)=e^{-\mu_2 x} e^{\lambda _2(\mu_2)t}\phi^*_2(t,x)
$$
is a solution of the linear equation
\begin{equation}\label{Eq1}
\frac{\partial u_2}{\partial t}=L_2u_2+(b_2(t,x)-a_{22}(t,x)u^*_2(t,x))u_2.
\end{equation}
Since $c_1>c_2$, it follows that the function
$$
v_2(t,x):=e^{-\mu_2(x-c_1t)}\phi^*_2(t,x)=
e^{\mu_2(c_1-c_2)t}u_2(t,x),\quad t\ge0,\ x\in\mathbb{R},
$$
satisfies
\begin{equation}\label{ineq}
\frac{\partial v_2}{\partial t}\ge L_2v_2+(b_2(t,x)-a_{22}(t,x)u^*_2(t,x))v_2.
\end{equation}

Define two wave-like functions:
\begin{equation}
\overline{u}_1(t,x):=\min\{m_0e^{-\mu_1(x-c_1t)}\phi^*_1(t,x),u^*_1(t,x)\},\quad t\ge0,\ x\in \mathbb{R},
\end{equation} and
\begin{equation}
\overline{u}_2(t,x):=\min\{q_0e^{-\mu_2(x-c_1t)}\phi^*_2(t,x),u^*_2(t,x)\},\quad t\ge0,\ x\in\mathbb{R},
\end{equation}
where
$$q_0:= \max_{(t,x)\in[0,\omega]\times[0,L]}\frac{u^*_2(t,x)}{\phi^*_2(t,x)}>0,\quad m_0:=\min_{(t,x)\in[0,\omega]\times[0,L]}\frac{q_0a_{22}(t,x)\phi^*_2(t,x)}{a_{21}(t,x)\phi^*_1(t,x)}>0.$$
Now, we are ready to verify that $(\overline{u}_1,\overline{u}_2)$ is an upper solution to system \eqref{NModel}. Indeed, for all $x-c_1 t>\frac{1}{\mu_1}\ln\frac{m_0\phi_1^*(t,x)}{u^*_1(t,x)}$, we have $\overline{u}_1(t,x)=m_0e^{-\mu_1(x-c_1t)}\phi^*_1(t,x)$, and hence,
\begin{eqnarray}
& &\frac{\partial \overline{u}_1}{\partial t}-L_1\overline{u}_1-\overline{u}_1(b_1(t,x)-a_{12}(t,x)u^*_2(t,x)-a_{11}(t,x)\overline{u}_1+a_{12}(t,x)\overline{u}_2)\nonumber\\
& &\ge\frac{\partial \overline{u}_1}{\partial t}-L_1\overline{u}_1-b_1(t,x)\overline{u}_1=0\nonumber.
\end{eqnarray}
For all $x-c_1t<\frac{1}{\mu_1}\ln\frac{m_0\phi_1^*(t,x)}{u^*_1(t,x)}$, we obtain $\overline{u}_1(t,x)=u^*_1(t,x)$, and hence,
\begin{eqnarray*}
& &\frac{\partial \overline{u}_1}{\partial t}-L_1\overline{u}_1-\overline{u}_1(b_1(t,x)-a_{12}(t,x)u^*_2(t,x)-a_{11}(t,x)\overline{u}_1+a_{12}(t,x)\overline{u}_2)\\
& &\ge\frac{\partial \overline{u}_1}{\partial t}-L_1\overline{u}_1-\overline{u}_1(b_1(t,x)-a_{11}(t,x)\overline{u}_1)=0.
\end{eqnarray*}
On the other hand, for all $x-c_1t\!>\!\frac{1}{\mu_2}\ln\frac{q_0\phi_2^*(t,x)}{u^*_2(t,x)}\!>\!0$, it follows that  $$\overline{u}_2(t,x)=q_0e^{-\mu_2(x-c_1t)}\phi^*_2(t,x),$$ which satisfies inequality  \eqref{ineq}. Note that $$\overline{u}_1(t,x)\le m_0e^{-\mu_1(x-c_1t)}\phi^*_1(t,x), \quad \forall t\ge0,\ x\in\mathbb{R},$$ and $\mu_2\in(0,\mu_1)$, we get
$$
\begin{array}{l}
\frac{\partial \overline{u}_2}{\partial t}-L_2\overline{u}_2-a_{21}(t,x)(u^*_2(t,x)-\overline{u}_2)\overline{u}_1-\overline{u}_2(b_2(t,x)-2a_{22}(t,x)u^*_2(t,x)+a_{22}(t,x)\overline{u}_2)\nonumber\\
=\frac{\partial \overline{u}_2}{\partial t}-L_2\overline{u}_2-(b_2(t,x)-a_{22}(t,x)u^*_2(t,x))\overline{u}_2+(u^*_2(t,x)-\overline{u}_2)(a_{22}(t,x)\overline{u}_2-a_{21}(t,x)\overline{u}_1)\\
\ge(u^*_2(t,x)-\overline{u}_2)e^{-\mu_1(x-c_1t)}a_{21}(t,x)\phi^*_1(t,x)(\frac{q_0a_{22}(t,x)\phi^*_2(t,x)}{a_{21}(t,x)\phi^*_1(t,x)}-m_0)\\
\ge0.
\end{array}$$
For all $x-c_1t<\frac{1}{\mu_2}\ln\frac{q_0\phi_2^*(t,x)}{u^*_2(t,x)}$, we have $\overline{u}_2(t,x)=u^*_2(t,x)$. Therefore,
\begin{align*}
& \frac{\partial \overline{u}_2}{\partial t}-L_2\overline{u}_2-a_{21}(t,x)(u^*_2(t,x)-\overline{u}_2)\overline{u}_1-\overline{u}_2(b_2(t,x)-2a_{22}(x)u^*_2(t,x)+a_{22}(t,x)\overline{u}_2)\\
&=\frac{\partial u^*_2}{\partial t}-L_2u^*_2-u^*_2(b_2(t,x)-a_{22}(t,x)u^*_2)=0.
\end{align*}
It then follows that $\overline{u}=(\overline{u}_1,\overline{u}_2)$ is a continuous upper  solution of system \eqref{NModel}.
\smallskip

  Let $\phi\in \mathcal{C}_{\beta(0,\cdot)}$  with $\phi(x)\ge \sigma$, $\forall x\le K$ and $\phi(x)=0$, $\forall x\ge H$, for some $\sigma\in \mathbb{R}^2$ with $\sigma\gg0$ and $K,H\in\mathbb{R}$. By the arguments in \cite[Lemma 2.2]{WLL2002} and the proof of Theorem \ref{m-SS}, as applied to $\tilde{P}_\omega$,  it follows that for any $c<\bar{c}_+$, there exists $\delta(c)>0$ such that
  \begin{equation}\label{ineq2}{\liminf}_{n\rightarrow\infty, x\le cn\omega}|u(n\omega,x,\phi)|\ge\delta(c)>0.\end{equation} Moreover, there exists a sufficiently large positive constant $A\in L\mathbb{Z}$ such that
$$\phi(x)\le\overline{u}(0,x-A) :=\psi(x),\quad \forall x\in \mathbb{R}.$$
By the translation invariance of $Q_t$, it follows that $\overline{u}(t,x-A)$ is still an upper solution of system \eqref{NModel}, and hence,
\begin{equation}\label{ineq3}0\le u(t,x,\phi)\le u(t,x,\psi)=\overline{u}(t,x-A),\quad \forall x\in\mathbb{R},\  t\ge0.\end{equation}
Fix a number $\hat c\in(c_1,\bar{c}_+)$. Letting $t=n\omega$, $x=\hat cn\omega$ and $n\rightarrow\infty$ in \eqref{ineq3}, together with \eqref{ineq2}, we have
$$0<\delta(\hat c)\le\liminf_{n\rightarrow\infty}|u(n\omega,\hat{c}n\omega,\phi)|\le\lim_{n\rightarrow\infty}|\overline{u}(n\omega,\hat cn\omega-A)|=0,$$
which is a contradiction. Thus, $c^*_+=\bar{c}_+$.
\end{proof}

Note that the leftward case can be addressed in a similar way. Indeed, by making a change of variable $v(t,x)=u(t,-x)$ for system \eqref{NModel}, we obtain similar results for the rightward case of the resulting system, which is the leftward case for system \eqref{NModel}.
\begin{remark}\label{Re1}
In the case where $L_iu=\frac{\partial}{\partial x}(d_i(x)\frac{\partial u}{\partial x})$ with $d_i\in C^{1+\nu}(\mathbb{R})$ in system \eqref{NModel}, $i=1,2$, or all the coefficient functions of system \eqref{NModel} are even in x except $g_i$ is odd in x, $i=1,2$, it follows from Lemma  \ref{H45} that system \eqref{NModel} admits a single rightward spreading speed which is coincident with the minimal rightward wave speed provided that (H1)--(H3) hold.
\end{remark}
\subsection{Linear determinacy of spreading speed}
In this subsection, we give a set of sufficient conditions for the rightward spreading speed to be determined by the linearization of system \eqref{NModel} at $\hat E_2=(0,0)$, which is
\begin{align}\label{linear}
&\frac{\partial v_1}{\partial t}\!=\!L_1v_1+(b_1(t,x)-a_{12}(t,x)u^*_2(t,x))v_1, \\
&\frac{\partial v_2}{\partial t}\!=\!L_2v_2\!+\!a_{21}(t,x)u^*_2(t,x)v_1
+\!(b_2(t,x)\!-\!2a_{22}(t,x)u^*_2(t,x))v_2,\quad t>0,\  x\in\mathbb{R}\nonumber.
\end{align} 

Clearly, under (H2) the following scalar equation
\begin{equation}\label{c0eq}
\frac{\partial u}{\partial t}= L_1u+u(b_1(t,x)-a_{12}(t,x)u^*_2(t,x)-a_{11}(t,x)u),\quad t>0, x\in\mathbb{R}, \\
\end{equation}
admits a rightward spreading speed (also the minimal rightward wave speed) $c^0_+=\inf\limits_{\mu>0}\frac{\lambda_0(\mu)}{\mu}$, where $\lambda_0(\mu)$ is the principle eigenvalue of the following parabolic eigenvalue problem:
\begin{eqnarray}\label{eep0}
& & \lambda \psi=-\frac{\partial\psi}{\partial t}+d_1(t,x)\frac{\partial^2\psi}{\partial x^2}-(2\mu d_1(t,x)+g_1(t,x))\frac{\partial\psi}{\partial x}\\
& &+(d_1(t,x)\mu^2+g_1(t,x)\mu+b_1(t,x)-a_{12}(t,x)u^*_2(t,x))\psi,\quad (t,x)\in\mathbb{R}\times\mathbb{R},\nonumber\\
& &\psi(t,x+L)=\psi(t,x),\ \psi(t+\omega,x)=\psi(t,x),\quad (t,x)\in\mathbb{R}\times\mathbb{R}.\nonumber
\end{eqnarray}
The next result shows that $c^0_+$ is a lower bound of the slowest spreading $c^*_+$ of system \eqref{NModel}.
\begin{proposition}\label{lb}Let (H1)--(H3) hold. Then $c^*_+\ge c^0_+$.
\end{proposition}
\begin{proof}In the case where $\bar{c}_+>c^*_+$, by the same arguments as in Theorem \ref{Qspreading}, we see that $c^*_+\ge c^*_{1+}$ where $c^*_{1+}$ is the rightward spreading speed of \eqref{u1}. Since $b_1(t,x)>b_1(t,x)-a_{12}(t,x)u^*_2(t,x), \forall (t,x)\in\mathbb{R}\times\mathbb{R}$, by Lemma \ref{mp} (a) with $d(t,x)=d_1(t,x)$ and $g(t,x)=g_1(t,x)$, it is easy to see that $\lambda_1(\mu)>\lambda_0(\mu), \forall \mu\ge0$,  where  $\lambda_1(\mu)$ is the principal eigenvalue of \eqref{eep}. Thus, we have $c^*_+\ge c^*_{1+}>c^0_+$.

In the case where $\bar{c}_+=c^*_+$, let $u(t,\cdot,\phi)=(u_1(t,\cdot,\phi),u_2(t,\cdot,\phi))$ be the solution of system \eqref{NModel} with $u(0,\cdot)=\phi=(\phi_1,\phi_2)\in \mathcal{C}_{\beta(0,\cdot)}$. Then the positivity of the solution implies that
$$\frac{\partial u_1}{\partial t}\ge L_1u_1+u_1(b_1(t,x)-a_{12}(t,x)u^*_2(t,x)-a_{11}(t,x)u_1),\quad t>0, x\in\mathbb{R}.$$
Let $v(t,x,\phi_1)$ be the unique solution of \eqref{c0eq} with $v(0,\cdot)=\phi_1$. Then the comparison principle yields that
 \begin{equation}\label{ineq1}u_1(t,x,\phi)\ge v(t,x,\phi_1), \quad \forall t\ge0,\  x\in \mathbb{R}.\end{equation}
 Since $\lambda(d_1,g_1,b_1-a_{12}u^*_2)>0$, Proposition \ref{VLexistence} implies that there exists a unique positive time-space periodic solution $v_0(t,x)$ of \eqref{c0eq}. Let $\phi^0=(\phi_1^0,\phi_2^0)\in \mathcal{C}_{\beta(0,\cdot)}$
 be chosen as in Theorem \ref{Qspreading} (i) and (ii) such that  $\phi_1^0\leq  v_0(0,\cdot)$. Suppose, by contradiction,  $c^*_+<c^0_+$. Choose $\hat c\in(\bar{c}_+,c^0_+)$. Then Theorem  \ref{Qspreading} implies that
  $\lim_{t\rightarrow\infty,x\ge \hat ct}u_1(t,x,\phi^0)=0$. By Theorem \ref{m-SS}, as applied to system \eqref{c0eq}, we further obtain $\lim_{t\rightarrow\infty,x\le \hat ct}(v(t,x,\phi^0_1)-v_0(t,x))=0$.
However, letting $x=\hat c t$ in \eqref{ineq1},  we get  $\lim_{t\rightarrow\infty,x= \hat ct}(v(t,x,\phi^0_1))=0$, which is a contradiction.
\end{proof}

For any given $\mu\in \mathbb{R}$, letting $v(t,x)=e^{-\mu x}u(t,x)$ in \eqref{linear}, we then have
\begin{eqnarray}\label{uModel}
\frac{\partial u_1}{\partial t}&\!=&\!L_1u_1\!-\!2\mu d_1(t,x)\frac{\partial u_1}{\partial x}\!+\!(d_1(t,x)\mu^2\!+\!g_1(t,x)\mu\!+\!b_1(t,x)\!-\!a_{12}(t,x)u^*_2(t,x))u_1, \nonumber\\
\frac{\partial u_2}{\partial t}&\!=&\!L_2u_2\!-\!2\mu d_2(t,x)\frac{\partial u_2}{\partial x}+a_{21}(t,x)u^*_2(x)u_1\\
& &+(d_2(t,x)\mu^2\!+\!g_2(t,x)\mu\!+\!b_2(t,x)\!-\!2a_{22}(t,x)u^*_2(t,x))u_2,\quad t>0, x\in\mathbb{R}\nonumber.
\end{eqnarray}
Substituting $u(t,x)=e^{\lambda t}\phi(t,x)$ into \eqref{uModel}, we obtain the following periodic eigenvalue problem:
\begin{align}\label{Lpep}
\lambda \phi_1&=-\frac{\partial\phi_1}{\partial t}+d_1(t,x)\frac{\partial^2\phi_1}{\partial x^2}-(2\mu d_1(t,x)+g_1(t,x))\frac{\partial\phi_1}{\partial x}\nonumber\\
 &+(d_1(t,x)\mu^2+g_1(t,x)\mu+b_1(t,x)-a_{12}(t,x)u^*_2(t,x))\phi_1,\quad (t,x)\in\mathbb{R}\times\mathbb{R},\nonumber\\
\lambda \phi_2&=-\frac{\partial\phi_2}{\partial t}+d_2(t,x)\frac{\partial^2\phi_2}{\partial x^2}-(2\mu d_2(t,x)+g_2(t,x))\frac{\partial\phi_2}{\partial x}+a_{21}(t,x)u^*_2(t,x)\phi_1\nonumber\\
&  +\left(d_2(t,x)\mu^2\!+\!g_2(t,x)\mu\!+\!b_2(t,x)\!-\!2a_{22}(t,x)u^*_2(t,x)\right)\phi_2,\quad  (t,x)\in\mathbb{R}\times\mathbb{R},\nonumber\\
\phi_i(t,x&+L)=\phi_i(t,x),\ \phi_i(t+\omega,x)=\phi_i(t,x),\quad (t,x)\in\mathbb{R}\times\mathbb{R},\ i=1,2.
\end{align}
Let $\overline{\lambda}(\mu)$ be the principal eigenvalue of the following  periodic eigenvalue problem:
\begin{align}\label{Lpep2}
\lambda \psi&=-\frac{\partial\psi}{\partial t}+d_2(t,x)\frac{\partial^2\psi}{\partial x^2}-(2\mu d_2(t,x)+g_2(t,x))\frac{\partial\psi}{\partial x}\nonumber\\
&  +\left(d_2(t,x)\mu^2\!+\!g_2(t,x)\mu\!+\!b_2(t,x)\!-\!2a_{22}(t,x)u^*_2(t,x)\right)\psi,\quad  (t,x)\in\mathbb{R}\times\mathbb{R},\nonumber\\
\psi(t,x&+L)=\psi(t,x),\ \psi(t+\omega,x)=\psi(t,x),\quad (t,x)\in\mathbb{R}\times\mathbb{R}.
\end{align}
Since  there exists $\mu_0>0$ such that $c^0_+=\frac{\lambda_0(\mu_0)}{\mu_0}$. Now we introduce the following condition:
\begin{enumerate}
\item[(D1)]  $\lambda_0 (\mu_0)>\overline{\lambda}(\mu_0)$.
\end{enumerate}

\begin{proposition}\label{ef}
Let (H1)--(H3) and (D1) hold. Then the periodic eigenvalue problem \eqref{Lpep} with $\mu=\mu_0$ has a simple eigenvalue $\lambda_0(\mu_0)$ associated with a positive periodic eigenfunction $\phi^*=(\phi_1^*,\phi_2^*)$.
\end{proposition}
\begin{proof}
 Clearly, there exists a positive eigenfunction $\phi_1^*$ associated with the principle eigenvalue $\lambda_0(\mu_0)$ of \eqref{c0eq}. Since the first equation of \eqref{Lpep} is decoupled from the second one, it suffices to show that $\lambda_0(\mu_0)$ has a positive eigenfunction $\phi^*=(\phi^*_1,\phi^*_2)$ in \eqref{Lpep}, where $\phi^*_2$ is to be determined. Let $U(t,s)$, $0\le s<t$, be the evolution operator generated by \eqref{uModel} with $u(0,\cdot)\in \mathbb{P}$, and $U_1(t,s)$ and $U_2(t,s)$, $0\le s<t$ be the evolution operators generated by the following scalar parabolic equations:
\begin{align*}
\label{uModel2}
\frac{\partial u}{\partial t}=& L_1u-2\mu_0 d_1(t,x)\frac{\partial u}{\partial x}\\
&+(d_1(t,x)\mu_0^2+g_1(t,x)\mu_0+b_1(t,x)-a_{12}(t,x)u^*_2(t,x))u, \quad t>0, \ x\in\mathbb{R},\nonumber\\
 u(0,\cdot)&=\varphi_1\in \mc{P}
\end{align*}
and
\begin{align*}
\frac{\partial u}{\partial t}=& L_2u\!-\!2\mu_0 d_2(t,x)\frac{\partial u}{\partial x}\\
&+(d_2(t,x)\mu_0^2\!+\!g_2(t,x)\mu_0\!+\!b_2(t,x)\!-\!2a_{22}(t,x)u^*_2(t,x))u,\quad t>0,\  x\in\mathbb{R}\nonumber,\\
u(0,\cdot)&=\varphi_2\in \mc{P},
\end{align*}
respectively.
By the variation of constants formula for scalar parabolic equations, it then follows that
\begin{equation*}
U(t,0)\left(\begin{array}{l}
\varphi_1\\
\varphi_2
\end{array}\right)=\left(\begin{array}{c}U_1(t,0)\varphi_1\\ U_2(t,0)\varphi_2+\int^{t}_{0}U_2(t,s)a_{21}(s,\cdot)u^*_2(s,\cdot)U_1(s,0)\varphi_1ds
\end{array}\right),\quad \forall t>0.\end{equation*}
And it is easy to see that $U_1(\omega,0)$ and $U_2(\omega,0)$ are strongly positive and compact linear operators on $\mc{P}$. Let $r_1$ and $r_2$ be the spectral radii of $U_1(\omega,0)$ and $U_2(\omega,0)$. Then the Krein-Rutman theorem (see e.g., \cite[Theorem 7.2 ]{Hess}) implies $r_i$ is the principle eigenvalue of $U_i(\omega,0), i=1,2$, and $r_1=e^{\lambda_0(\mu_0)\omega}$ and $r_2=e^{\overline{\lambda}(\mu_0)\omega}$. Moreover, $U_1(t,0)\phi_1^*(0,\cdot)=e^{\lambda_0(\mu_0)t}\phi_1^*(t,\cdot)>0$, $\forall t>0.$
By \cite[Theorem 7.3]{Hess} and (D1), it follows that
\begin{equation}\label{1}(r_1-U_2(\omega,0))\varphi_2=\int^{\omega}_{0}U_2(\omega,s)a_{21}(s,\cdot)u^*_2(s,\cdot)U_1(s,0)\phi^*_1(0,\cdot)ds>0,\end{equation}
has a unique positive solution $\varphi^*_2\in \mc{P}$. Therefore,  $\varphi^*=(\phi_1^*(0,\cdot),\varphi^*_2)\in\mathbb{P}$ is a positive eigenfunction of $U(\omega,0)$ with the eigenvalue $r_1=e^{\lambda_0(\mu_0)\omega}$, that is, $U(\omega,0)\varphi^*=r_1\varphi^*$. Let $$\phi^*_2(t,\cdot)=e^{-\lambda_0(\mu_0)t}U_2(t,0)\varphi^*_2+\int^{t}_{0}e^{-\lambda_0(\mu_0)(t-s)}U_2(t,s)a_{21}(s,\cdot)u^*_2(s,\cdot)\phi^*_1(s,\cdot)ds.$$ Clearly, $\phi^*_2(t,x)$ is positive for $(t,x)\in\mathbb{R}_+\times\mathbb{R}$, and satisfies the second equation of \eqref{Lpep}, and $$\phi^*_2(\omega,\cdot)=e^{-\lambda_0(\mu_0)\omega}r_1\varphi^*_2=\varphi^*_2=\phi^*_2(0,\cdot).$$
This implies that $\phi^*=(\phi^*_1,\phi^*_2)$ is the positive time and space periodic eigenfunction associated with $\lambda_0(\mu_0)$.
 Since $\lambda_0(\mu_0)$ is a simple eigenvalue for \eqref{c0eq}, we see that so is $\lambda_0(\mu_0)$ for \eqref{Lpep}.
\end{proof}

From Proposition \ref{ef}, it is easy to see that for any given $M>0$, the function \begin{equation}\label{eqU}
U(t,x)=Me^{-\mu_0x}e^{\lambda_0(\mu_0)t}\phi^*(t,x),\quad t\ge0,\ x\in\mathbb{R},
\end{equation} is a positive solution of system \eqref{linear}. In order to obtain an explicit formula for the spreading speeding $\bar{c}_+$, we need the following additional condition:
\begin{enumerate}
\item[(D2)] $\frac{\phi^*_1(t,x)}{\phi^*_2(t,x)}\ge\max\left\{\frac{a_{12}(t,x)}{a_{11}(t,x)},\frac{a_{22}(t,x)}{a_{21}(t,x)}\right\},\quad \forall (t,x)\in\mathbb{R}\times\mathbb{R}$.
\end{enumerate}
Now we are in a position to show that system \eqref{NModel} admits a single rightward spreading speed $\bar{c}_+$, which is linearly determinate.

\begin{theorem}\label{c0}
Let (H1)--(H3) and (D1)--(D2) hold. Then $\bar{c}_+=c^*_+=c^0_+=\inf_{\mu>0}\frac{\lambda_0(\mu)}{\mu}$.
\end{theorem}
\begin{proof}
 First, we verify that $U(t,x)$, as defined in \eqref{eqU}, is an upper solution of system \eqref{NModel}. Since $\frac{U_1}{U_2}=\frac{\phi^*_1}{\phi^*_2}$ and (D2) holds true, it follows that
\begin{eqnarray}
& &\frac{\partial U_1}{\partial t}\!-\!L_1 U_1-U_1(b_1(t,x)-a_{12}(t,x)u^*_2(t,x)-a_{11}(t,x)U_1+a_{12}(t,x)U_2)\nonumber\\
& &=a_{11}(t,x)U_1U_2\left(\frac{U_1}{U_2}-\frac{a_{12}(t,x)}{a_{11}(t,x)}\right)\nonumber\\
& &=a_{11}(t,x)U_1U_2\left(\frac{\phi^*_1(t,x)}{\phi^*_2(t,x)}-\frac{a_{12}(t,x)}{a_{11}(t,x)}\right)\ge 0,
\end{eqnarray}
and
\begin{eqnarray}
& &\frac{\partial U_2}{\partial t}\!-\!L_2U_2\!-\!a_{21}(t,x)U_1(u^*_2(t,x)\!-\!U_2)
-\!U_2(b_2(t,x)\!-\!2a_{22}(t,x)u^*_2(t,x)\!+\!a_{22}(t,x)U_2).\nonumber\\
& &=a_{21}(t,x)U_2^{^2}\left(\frac{U_1}{U_2}-\frac{a_{22}(t,x)}{a_{21}(t,x)}\right)\nonumber\\
& &=a_{21}(t,x)U_2^{^2}\left(\frac{\phi^*_1(t,x)}{\phi^*_2(t,x)}-\frac{a_{22}(t,x)}{a_{21}(t,x)}\right)\ge0.
\end{eqnarray}
Thus, $U(t,x)$ is an upper solution of \eqref{NModel}. Choose some $\phi^0\in\mathcal{C}_{\beta(0,\cdot)}$ satisfying the conditions in Theorem \ref{Qspreading} (i) and (ii). Then there exists a sufficiently large number $M_0>0$ such that
$$0\le\phi^0(x)\le M_0e^{-\mu_0x}\phi^*(0,x)=U(0,x),\quad \forall x\in\mathbb{R}.$$
Let $W(t,x)$ be the unique solution of system \eqref{NModel} with $W(0,\cdot)=\phi^0$. Then the comparison principle, together with the fact that $c^0_+\mu_0=\lambda_0(\mu_0)$, leads that
$$0\!\le\! W(t,x)\!\le\! U(t,x)\!=\!M_0e^{-\mu_0x}e^{\lambda_0(\mu_0)t}\phi^*(t,x)\!=\!M_0e^{-\mu_0(x-c^0_+t)}\phi^*(t,x),\quad \forall t\ge0,\ x\in\mathbb{R}.$$
It follows that for any given $\epsilon>0$, there holds
$$0\le W(t,x)\le M_0e^{-\mu_0\epsilon t}\phi^*(t,x),\quad \forall t\ge0,\ x\ge(c^0_++\epsilon)t,$$
and hence,
$$\lim_{t\rightarrow\infty,x\ge (c^0_++\epsilon)t}W(t,x)=0.$$
By Theorem \ref{Qspreading} (ii), we obtain $c^*_+\le c^0_++\epsilon$. Letting $\epsilon\rightarrow 0$, we have $c^*_+\le c^0_+$. In the case that $\bar{c}_+>c^*_+$, the proof of Proposition \ref{lb} shows that $c^*_+>c^0_+$, a contradiction. This implies that $\bar{c}_+=c^*_+=c^0_+$.
\end{proof}
To finish this section, we consider the following time-periodic Lotka--Volterra competition model \cite{ZR}: \begin{eqnarray}\label{VL2}
& &\frac{\partial u_1}{\partial t}=\frac{\partial ^2u_1}{\partial x^2}+u_1(b_1(t)-a_{11}(t)u_1-a_{12}(t)u_2), \\
& &\frac{\partial u_2}{\partial t}=d\frac{\partial ^2u_2}{\partial x^2}+u_2(b_2(t)-a_{21}(t)u_1-a_{22}(t)u_2),\quad t>0, \ x\in\mathbb{R}. \nonumber
\end{eqnarray}
Here $d>0$ and all other coefficient functions are positive and $\omega$-periodic in $t$.

For convenience, define $\overline{w}=\frac{1}{\omega}\int_{0}^{\omega}w(t)dt$ with any $\omega$-periodic function $w(t)$. We first make the following assumption (see (A2) in \cite{ZR}):
\begin{enumerate}
	\item [(P1)] $\overline{b}_1>\max\limits_{t\in[0,\omega]}\frac{a_{12}(t)}{{a}_{22}(t)}\cdot\overline{b}_2>0$, and $0<\overline{b}_2\le\max\limits_{t\in[0,\omega]}\frac{ a_{21}(t)}{a_{11}(t)}\cdot\overline b_1.$

\end{enumerate}
It is easy to see that if (P1) holds, then Proposition \ref{c}  implies that (H3) is valid. A straightforward computation shows that $\lambda(1,0,b_1)=\overline{b}_1>0$ and $\lambda(1,0,b_2)=\overline{b}_2>0$. Thus,  (H1) holds true and system \eqref{VL2} admits three time periodic solutions $(0,0)$, $(u^*_1(t),0)$, and $(0,u^*_2(t))$. Moreover, we can show that $$\lambda(1,0,b_1-a_{12}u^*_2)=\overline{b_1-a_{12}u^*_2}\ge\overline{b}_1-\max\limits_{t\in[0,\omega]}\frac{a_{12}(t)}{{a}_{22}(t)}\cdot\overline{a_{22}u_2^*}=\overline{b}_1-\max\limits_{t\in[0,\omega]}\frac{a_{12}(t)}{{a}_{22}(t)}\cdot\overline{b_2}>0,$$
 and hence, (H2) ia also valid. (H4) and (H5) are automatically satisfied since all coefficient functions are independent of $x$ (treated as even functions of $x$). It then follows that system \eqref{VL2} admits a single spreading speed (also the minimal wave speed) $\bar{c}_+$ no matter whether it is linearly determinate.
Next, we make another assumption (see \cite[Theorem 2.5]{ZR}):
\begin{enumerate}
	\item [(P2)] $0<d\le1$, $a_{11}(t)u^*_1(t)-a_{12}(t)u^*_2(t)\ge a_{21}(t)u^*_1(t)-a_{22}(t)u^*_2(t)\ge0, \forall t\in\mathbb{R}.$
	\end{enumerate}
In what follows, we show that (P2) is sufficient for (D1) and (D2) to hold.
Clearly, $\lambda_0(\mu)$ and $\overline{\lambda}(\mu)$ become the principal eigenvalues of the following  periodic eigenvalue problems:
\begin{eqnarray}\label{ep0}
& & \lambda \psi=-\frac{d\psi}{d t}+(\mu^2+b_1(t)-a_{12}(t)u^*_2(t))\psi,\quad t\in\mathbb{R},\nonumber\\
& &\psi(t+\omega)=\psi(t),\quad t\in\mathbb{R},\nonumber
\end{eqnarray}
and
\begin{eqnarray}\label{pep2}
& &\lambda \psi=-\frac{d\psi}{d t}+\left(d\mu^2+b_2(t)-2a_{22}(t)u^*_2(t)\right)\psi,\quad  t\in\mathbb{R},\nonumber\\
& &\psi(t+\omega)=\psi(t),\quad t\in\mathbb{R},
\end{eqnarray}
respectively.
It is easy to see that
$$\lambda_0(\mu)=\mu^2+\overline{b_1-{a}_{12}u^*_2},\quad\overline{\lambda}(\mu)=d\mu^2+\overline{b_2-2{a}_{22}u^*_2}.$$
By virtue of  (P2) and
$$
c^0_+=\inf\limits_{\mu>0}\frac{\lambda_0(\mu)}{\mu}=\inf\limits_{\mu>0}\left\{\mu+\frac{\overline{b_1-a_{12}u^*_2}}{\mu}\right\},
$$
it follows that $$c^0_+=2\sqrt{\overline{b_1-a_{12}u^*_2}}>0,\ \mu_0=\sqrt{\overline{b_1-a_{12}u^*_2}}.$$
We then see from (P2) that (D1) holds true.

Let $(\phi_1(t),\phi_2(t))$ be a positive  eigenfunction, associated with  $\lambda_0(\mu_0)$, of the following eigenvalue problem:
\begin{align}\label{pep}
\lambda \phi_1&=-\frac{d\phi_1}{d t}+(\mu_0^2+b_1(t)-a_{12}(t)u^*_2(t))\phi_1,\quad t\in\mathbb{R},\nonumber\\
\lambda \phi_2&=-\frac{d\phi_2}{dt}+a_{21}(t)u^*_2(t)\phi_1+\left(d\mu_0^2+b_2(t)-2a_{22}(t)u^*_2(t)\right)\phi_2,\quad  t\in\mathbb{R},\nonumber\\
\phi_i(t+\omega)&=\phi_i(t),\quad t\in\mathbb{R},\ i=1,2.
\end{align}
Next we verify that $\phi_2(t)\le\frac{u^*_2(t)}{u^*_1(t)}\phi_1(t):=v(t),\forall t\in\mathbb{R}$.  Note that
\begin{align*}
& -\frac{dv}{dt}+a_{21}(t)u^*_2(t)\phi_1(t)+(d\mu_0^2+b_2(t)-2a_{22}(t)u^*_2(t)-2\mu^2_0)v\\
& =\frac{u^*_2(t)\phi_1(t)}{u^*_1(t)}\left[-\frac{u^*_1(t)}{u^*_2(t)}\left(\frac{u^*_2(t)}{u^*_1(t)}\right)^\prime-\frac{\phi_1^\prime(t)}{\phi_1(t)}+a_{21}(t)u^*_1(t)+d\mu_0^2\right.\\&\left.+b_2(t)-2a_{22}(t)u^*_2(t)-2\mu^2_0\right]\\
&=\frac{u^*_2(t)\phi_1(t)}{u^*_1(t)}\left[-\frac{u^*_1(t)}{u^*_2(t)}\left(\frac{u^*_2(t)}{u^*_1(t)}\right)^\prime+b_2(t)-b_1(t)-a_{22}(t)u^*_2(t)+a_{11}(t)u^*_1(t)\right.\\
&\left.+(d-1)\mu_0^2-a_{11}(t)u^*_1(t)+a_{12}(t)u^*_2(t)+a_{21}(t)u^*_1(t)-a_{22}(t)u^*_2(t)\right]\\
&\le\frac{u^*_2(t)\phi_1(t)}{u^*_1(t)}\left[-\frac{u^*_1(t)}{u^*_2(t)}\left(\frac{u^*_2(t)}{u^*_1(t)}\right)^\prime+b_2(t)-b_1(t)-a_{22}(t)u^*_2(t)+a_{11}(t)u^*_1(t)\right]\\
&=0.
\end{align*}
In view of the comparison principle and
the periodicity of $\phi_2(t)$ and $v(t)$,
it then suffices to show that  $\phi_2(t_0)\leq v(t_0)$ for some $t_0\in \mathbb{R}$.  Assume, by contradiction, that
$\phi_2(t)>v(t),\, \forall t\in \mathbb{R}$.
Thus, we have  $w(t):=v(t)-\phi_2(t)<0$, and
$$\frac{dw(t)}{dt}+\left(2\mu^2_0-d\mu_0^2-b_2(t)+2a_{22}(t)u^*_2(t)\right)w(t)\ge0, \, \, \forall t\in \mathbb{R}.$$
This implies that
$$\int_{0}^{\omega}(2\mu^2_0-d\mu_0^2-b_2(t)+2a_{22}(t)u^*_2(t))dt\le 0.$$
On the other hand,
we know that $\overline{b}_2=\overline{a_{22}u^*_2}$, and
$$\frac{1}{\omega}\int_{0}^{\omega}(2\mu^2_0-d\mu_0^2-b_2(t)+2a_{22}(t)u^*_2(t))dt=(2-d)\mu^2_0+\overline{a_{22}u^*_2}>0,$$
which leads to a contradiction.
This shows that $\frac{\phi_1(t)}{\phi_2(t)}
\ge\frac{u^*_1(t)}{u^*_2(t)}, \,
\forall t\in \mathbb{R}$.
It then follows from (P2) that
$$\frac{\phi_1(t)}{\phi_2(t)}\ge\frac{u^*_1(t)}{u^*_2(t)}\ge\max\left\{\frac{a_{12}(t)}{a_{11}(t)},\frac{a_{22}(t)}{a_{21}(t)}\right\},\quad \forall t\in\mathbb{R},$$
which implies  that (D2) is also valid. Therefore, if (P1) and (P2) hold, then
the spreading speed $\bar{c}_+$ is linearly determinate, equal to $2\sqrt{\overline{b_1-a_{12}u^*_2}}$.
\begin{remark}
Consider a more general reaction-diffusion competition system in a periodic habitat, that is,
\begin{eqnarray}\label{GVL}
   & &\frac{\partial u_1}{\partial t}=L_1u_1+u_1f_1(t,x,u_1,u_2),\\
   & &\frac{\partial u_2}{\partial t}=L_2u_2+u_2f_2(t,x,u_1,u_2),\quad t\in\mathbb{R}, \ x\in\mathbb{R}, \nonumber
   \end{eqnarray}
where the operator $L_i:=a^{(i)}_2(t,x)\frac{\partial^2}{\partial x^2}+a^{(i)}_1(t,x)\frac{\partial}{\partial x}$ with $a^{(i)}_2(t,x)>0,\forall (t,x)\in \mathbb{R}\times\mathbb{R}$, i.e., $L_i$ is uniformly elliptic, $i=1,2$. Assume that $a^{(i)}_j(t,x)$ and $f_i(t,x,u_1,u_2)$ are periodic in $t$ and $x$ with the same periods, respectively, H\"{o}lder continuous in $x$ of order $\nu\in(0,1)$ and in $t$ of order $\frac{\nu}{2}$, $1\leq i,j\leq 2$, and $f_i(t,x,u_1,u_2)$ are differentiable with respect to $u_1$ and $u_2$, $i=1,2$. Moreover, $\partial_{u_1}f_1(t,x,u_1,0)<0$ and   $\partial_{u_2}f_2(t,x,0,u_2)<0$, $\forall (t,x)\in\mathbb{R}\times\mathbb{R}$, $u_1\in\R_+$, $u_2\in\R_+$, and there exist $M_1>0$ and $M_2>0$ such that $f_1(t,x,M_1,0)\le0$, $f_2(t,x,0,M_2)\le 0$, $\partial_{u_2}f_1(t,x,u_1,u_2)<0$ and $\partial_{u_1}f_2(t,x,u_1,u_2)<0$ for all $(t,x,u_1,u_2)\in\mathbb{R}\times\mathbb{R}\times[0,M_1]\times[0,M_2]$. Then we can obtain analogous results on traveling waves and spreading speeds under similar assumptions to (H1)--(H5)
and (D1)--(D2).

\end{remark}
\subsection{An example}
In this section, we study the time periodic version of a well-known reaction diffusion model \cite{HMP}:
 \begin{eqnarray}\label{dhmp}
   & &\frac{\partial u_1}{\partial t}=d_1 \Delta u_1+u_1(a_\omega(t,x)-u_1-u_2), \\
   & &\frac{\partial u_2}{\partial t}=d_2\Delta u_2+u_2(a_\omega(t,x)-u_1-u_2),\quad t>0, \ x\in\mathbb{R}, \nonumber
   \end{eqnarray}
where $0<d_1<d_2$, $a_\omega(t,x)=a(t/\omega,x)$ and $a(t,x)$ is a continuous function on  $\mathbb{R}\times\mathbb{R}$ and it is $1$-periodic in $t$ and $L$-periodic in $x$.

For convenience, we use the same notations as in sections 2 and 3.  We first present some results on the principle eigenvalue $\lambda_m(\mu)$ of \eqref{geep}.
\begin{lemma}\label{mp}
Assume that time and space periodic functions $d,g,m\in C^{\frac{\nu}{2},\nu}(\mathbb{R}\times\mathbb{R})(\nu\in(0,1))$. Let $\lambda_m(\mu)(\mu\in\mathbb{R})$ be the principle eigenvalue of the following parabolic eigenvalue problem:
\begin{eqnarray}\label{geep}
& & \lambda \psi=-\frac{\partial\psi}{\partial t}+d(t,x)\frac{\partial^2\psi}{\partial x^2}-(2\mu d(t,x)+g(t,x))\frac{\partial\psi}{\partial x}\nonumber\\
& &+(d(t,x)\mu^2+g(t,x)\mu+m(t,x))\psi,\quad (t,x)\in\mathbb{R}\times\mathbb{R},\\
& &\psi(t,x+L)=\psi(t,x),\ \psi(t+\omega,x)=\psi(t,x),\quad (t,x)\in\mathbb{R}\times\mathbb{R}.\nonumber
\end{eqnarray}
Then the following statements are valid:
\begin{enumerate}
\item[(a)]If $m_1(t,x)\ge m_2(t,x)$ with $m_1(t,x)\not\equiv m_2(t,x),\forall (t,x)\in[0,\omega]\times[0,L]$, then $\lambda_{m_1}(\mu)>\lambda_{m_2}(\mu)$, $\forall\mu\in\mathbb{R}$.
\item[(b)] $\lambda_m(\mu)$ is a convex function of $\mu$ on $\mathbb{R}$.
\item[(c)]If either $d,m$ are even in $x$ and $g$ is odd in $x$, or $d\in C^{\frac{\nu}{2},1+\nu}(\mathbb{R}\times\mathbb{R})(\nu\in(0,1))$ and $g(t,x)=-\frac{\partial d(t,x)}{\partial x},\forall (t,x)\in\mathbb{R}\times\mathbb{R}$ and $d,m$ are even in $t$, then $\lambda_m(\mu)=\lambda_m(-\mu), \forall \mu\in\mathbb{R}$.
\end{enumerate}
\end{lemma}
\begin{proof}
By similar arguments to those in \cite[Lemma 15.5]{Hess} 
, it is easy to prove that (a) holds. (b) follows from the same arguments as in  \cite{Liang}.

In the case where $d,m$  are even functions in $x$ and $g$ is odd in $x$. Let $\psi(t,x) $ be eigenfunction
associated with $\lambda_m (\mu)$ .
Set $\phi(t,x)=\psi(t,-x)$, we then have \begin{eqnarray*}
& & \lambda \phi=-\frac{\partial\phi}{\partial t}+d(t,-x)\frac{\partial^2\phi}{\partial x^2}+(2\mu d(t,-x)+g(t,-x))\frac{\partial\phi}{\partial x}\nonumber\\
& &+(d(t,-x)\mu^2+g(t,-x)\mu+m(t,-x))\phi,\quad (t,x)\in\mathbb{R}\times\mathbb{R}.
\end{eqnarray*} Since $d(t,x)=d(t,-x), m(t,x)=m(t,-x), g(t,x)=-g(t,-x), \forall (t,x)\in\mathbb{R}\times\mathbb{R}$,
we obtain
\begin{eqnarray*}
	& & \lambda \phi=-\frac{\partial\phi}{\partial t}+d(t,x)\frac{\partial^2\phi}{\partial x^2}+(2\mu d(t,x)-g(t,x))\frac{\partial\phi}{\partial x}\nonumber\\
	& &+(d(t,x)\mu^2-g(t,x)\mu+m(t,x))\phi,\quad (t,x)\in\mathbb{R}\times\mathbb{R}.
\end{eqnarray*}
By the uniqueness of the principal eigenvalue, it follows that  $\lambda_m(-\mu)=\lambda_m(\mu), \forall \mu\in\mathbb{R}$.

 In the case where $d\in C^{\frac{\nu}{2},1+\nu}(\mathbb{R}\times\mathbb{R})(\nu\in(0,1))$ and $g(t,x)=-\frac{\partial d(t,x)}{\partial x},\forall (t,x)\in\mathbb{R}\times\mathbb{R}$, $d$, $m$ is even in $t$, for any given $\mu\in\mathbb{R}$, it is easy to see that  $\lambda_m(\mu)$ is also the principle eigenvalue of \begin{eqnarray}\label{geep2}
 & & \lambda \psi=\frac{\partial\psi}{\partial t}+\frac{\partial }{\partial x}\left(d(t,x)\frac{\partial\psi}{\partial x}\right)-2\mu d(t,x)\frac{\partial\psi}{\partial x}\nonumber\\
 & &+(d(t,x)\mu^2-\frac{\partial d(t,x)}{\partial x}\mu+m(t,x))\psi,\quad (t,x)\in\mathbb{R}\times\mathbb{R},\\
 & &\psi(t,x+L)=\psi(t,x),\ \psi(t+\omega,x)=\psi(t,x),\quad
  (t,x)\in\mathbb{R}\times\mathbb{R}.\nonumber
  \end{eqnarray}
 Let $\varphi(t,x)$ and $\phi(t,x)$ be the positive periodic eigenfunctions
associated with  $\lambda_m(\mu)$ and $\lambda_m(-\mu)$, respectively, and $\psi(t,x)=\varphi(-t,x), \forall (t,x)\in\mathbb{R}\times\mathbb{R}$. Note that $d,m$ are even in $t$, so is $\frac{\partial d(t,x)}{\partial x}$, it then follows that,
$$
\lambda_m(\mu) \psi=\frac{\partial\psi}{\partial t}+\frac{\partial }{\partial x}\left(d(t,x)\frac{\partial\psi}{\partial x}\right)-2\mu d(t,x)\frac{\partial\psi}{\partial x}+(d(t,x)\mu^2-\frac{\partial d(t,x)}{\partial x}\mu+m(t,x))\psi
$$
and
$$
\lambda_m(-\mu) \phi=-\frac{\partial\phi}{\partial t}+\frac{\partial }{\partial x}\left(d(t,x)\frac{\partial\phi}{\partial x}\right)+2\mu d(t,x)\frac{\partial\phi}{\partial x}+(d(t,x)\mu^2+\frac{\partial d(t,x)}{\partial x}\mu+m(t,x))\phi
$$
Using integration by parts, we have $$\int^{\omega}_{0}\int^{L}_{0}\frac{\partial\psi(t,x)}{\partial t}\phi(t,x) dxdt=-\int^{\omega}_{0}\int^{L}_{0}\frac{\partial \phi(t,x)}{\partial t}\psi(t,x)dxdt,$$
$$\int^{\omega}_{0}\int^{L}_{0}\frac{\partial}{\partial x}\left(d(t,x)\frac{\partial \psi(t,x)}{\partial x}\right)\phi(t,x) dxdt=\int^{\omega}_{0}\int^{L}_{0}\frac{\partial}{\partial x}\left(d(t,x)\frac{\partial\phi(t,x)}{\partial x}\right)\psi(t,x)dxdt,$$ and
\begin{small}\begin{align*}-&\mu\int^\omega_0\int^L_0\left[2d(t,x)\frac{\partial\psi(t,x)}{\partial x}\phi(t,x)+\frac{\partial d(t,x)}{\partial x}\psi(t,x)\phi(t,x)\right]dxdt\\
&=\mu \int^{\omega}_0\int^L_0\left[2\frac{\partial d(t,x)\phi(t,x)}{\partial x}\psi(t,x)-\frac{\partial d(t,x)}{\partial x}\psi(t,x)\phi(t,x)\right]dxdt\\
&=\mu\int^\omega_0\int^L_0\left[2d(t,x)\frac{\partial\psi(t,x)}{\partial x}\phi(t,x)+\frac{\partial d(t,x)}{\partial x}\psi(t,x)\phi(t,x)\right]dxdt.
\end{align*}
\end{small}
It then follows that \begin{eqnarray}\label{22}
\lambda_m (\mu)\int_0^\omega\int_0^L\psi(t,x)\phi(t,x)dxdt=\lambda_m (-\mu)\int_0^\omega\int_0^L\phi(t,x)\psi(t,x)dxdt.
\end{eqnarray} Since $\int_0^\omega\int_0^L\phi(t,x)\psi(t,x)dxdt>0$, we have $\lambda_m(\mu)=\lambda_m(-\mu), \forall \mu\in\mathbb{R}$.
\end{proof}
\begin{lemma}\label{H45}
Assume that (H1) and (H2) hold. Then (H4) and (H5) are valid  provided that either all the coefficient functions of system \eqref{NModel} are even in $x$ except $g_i$ is odd in $x$, or all the cofficient functions of system \eqref{NModel} are independent of $t$,  $d_i\in C^{1+\nu}(\mathbb{R})(\nu\in(0,1))$ and $g_i(t,x)=-d_i'(x),\forall (t,x)\in\mathbb{R}\times\mathbb{R}, i=1,2.$
\end{lemma}
\begin{proof}
First, we prove that (H4) holds. Indeed, in either case, by Lemma \ref{mp}(c) with $m(t,x)=b_1(t,x)$ and $d(t,x)=d_1(t,x)$, it is easy to see that the principle $\lambda_1(\mu)$ of \eqref{eep} is an even function of $\mu$ on $\mathbb{R}$. Since $\lambda_1(\mu)$ is convex on $\mathbb{R}$ and $\lambda_1(0)>0$, we have $\lambda_1(\mu)>0,\forall \mu>0.$ It follows that $c_{1+}^*=\inf_{\mu>0}\frac{\lambda _1(\mu)}{\mu}>0$. Similarly, we can show that $c_{2-}^*>0$, this implies $c_{1+}^*+c_{2-}^*>0$.

To verify (H5), it suffices to show that $\lim_{\mu\to 0^+}\frac{\lambda _2(\mu)}{\mu}=0$, where $\lambda_2(\mu)$ is the principal eigenvalue of \eqref{eep2}. In the case where all the coefficient functions of \eqref{NModel} are even in $x$ except $g_i$ is odd in $x$, $i=1,2$, we have
$$\frac{\partial u^*_2}{\partial t}=d_2(t,x)\frac{\partial^2 u^*_2}{\partial x^2}+g_2(t,x)\frac{\partial u^*_2}{\partial x}+u^*_2(b_2(t,x)-a_{22}(t,x)u^*_2), \quad (t,x)\in\mathbb{R}\times\mathbb{R}.$$
Let $u^0_2(t,x)=u^*_2(t,-x)$. Since $d_2$, $b_2$, $a_{22}$ are even in $x$ and $g_2$ is odd in $x$, it follows that
$$\frac{\partial u^0_2}{\partial t}=d_2(t,x)\frac{\partial^2 u^0_2}{\partial x^2}+g_2(t,x)\frac{\partial u^0_2}{\partial x}+u^0_2(b_2(t,x)-a_{22}(t,x)u^0_2), \quad (t,x)\in\mathbb{R}\times\mathbb{R}.$$
This implies that $u^0_2(t,-x)$ is also a  time and space periodic positive solution for scalar equation \eqref{VLseq} with $d(t,x)=d_2(t,x)$, $g(t,x)=g_2(t,x)$, $c(t,x)=b_2(t,x)$ and $e(t,x)=a_{22}(t,x), \forall (t,x)\in\mathbb{R}\times\mathbb{R}$. In view of Proposition \ref{VLexistence}, the uniqueness of the time and space periodic positive solution implies that $u^*_2(t,-x)=u^*_2(t,x),\forall (t,x)\in\mathbb{R}\times\mathbb{R}$. Taking $d(
t,x)=d_2(t,x)$, $m(t,x)=b_2(t,x)-a_{22}(t,x)u^*_2(t,x)$, and $g(t,x)=g_2(t,x)$ in \eqref{geep}, we see from the former case in Lemma \ref{mp}(c) that $\lambda_2(\mu)$ is an even function on $\mathbb{R}$, and hence, $\lambda'_2(0)=0$. Since $\lambda_2(0)=0$, it follows that
$\lim_{\mu\rightarrow 0^+}\frac{\lambda _2(\mu)}{\mu}=\lambda_2'(0)=0<c^*_{1+}.$

 In the case where all the coefficient functions of system \eqref{NModel} are independent of $t$,  $d_i\in C^{1+\nu}(\mathbb{R})(\nu\in(0,1))$ and $g_i(t,x)=-d_i'(x),\forall (t,x)\in\mathbb{R}\times\mathbb{R}, i=1,2$, it easily follows from the latter case in  Lemma \ref{mp}(c) or the proof of \cite[lemma 5.2]{YZ}.
\end{proof}
Now we  introduce the following assumptions on system \eqref{dhmp}:
\begin{enumerate}
\item[(M)] $a(t,x)$ is non-trivial and even in $x$, and $\overline{a}=\frac{1}{L}\int_{0}^{1}\int_{0}^{L}a(t,x)dxdt\ge0$.
\end{enumerate}
\begin{lemma}\label{M}
Let (M) hold. Then (H1)--(H3) are valid for system \eqref{dhmp} if either of the following holds:
\begin{enumerate}
	\item[(a)]  $d_2$ is large enough;
	\item[(b)] $\omega$ is small enough.
\end{enumerate}
\end{lemma}
\begin{proof} Since we consider the periodic initial value problem. We may regard system \eqref{dhmp} as in the following system:
	\begin{eqnarray}\label{hmp}
	& &\frac{\partial u_1}{\partial t}=d_1 \Delta u_1+u_1(a_\omega(t,x)-u_1-u_2), \\
	& &\frac{\partial u_2}{\partial t}=d_2\Delta u_2+u_2(a_\omega(t,x)-u_1-u_2),\quad t>0, \ x\in(0,L), \nonumber\\
	& & u_i(0,x)=\phi_i(x)\in\overline{X}:=\{\phi\in C([0,L],\mathbb{R}): \phi(0)=\phi(L)\}, i=1,2.\nonumber
	\end{eqnarray}
Let $\phi(t,x)$ be the positive time-space periodic eigenfunction associated with the principal eigenvalue $\lambda(d_1,0,a)$, that is,
$$-\phi_t+d_1\phi_{xx}+a_\omega(t,x)\phi=\lambda(d_1,0,a)\phi, \quad (t, x)\in\mathbb{R}\times\mathbb{R}.$$
Dividing the above equation by $\phi$ and integrating by parts on $[0,L]\times[0,\omega]$, we get
$$\lambda(d_1,0,a)=\frac{1}{ L}\int_{0}^{1}\int_{0}^{L}a(t,x)dxdt+\frac{d_1}{\omega L}\int_{0}^{\omega}\int_{0}^{L}\left[\frac{ \phi_x(t,x)}{\phi(t,x)}\right]^2dxdt.$$
Since $a(t,x)$ is non-trivial in $x$, a simple computation shows that $\phi(t,x)$ is also non-trivial in $x$. Therefore, we have $$\lambda(d_1,0,a)>\frac{1}{ L}\int_{0}^{1}\int_{0}^{L}a(t,x)dxdt=\overline{a}\ge0.$$
Similarly, we can show that $\lambda(d_2,0,a)>0$.
It follows that (H1) holds provided that (M) is valid.

In the case where $d_2$ is large enough,  let $A_{d_2}$ denote the unbounded closed operator on $\overline{X}$ with the maximum norm defined by
$$D(A_{d_2})=\{u:u,u',u''\in\overline{X} \}, \quad A_{d_2}u=d_2u'', \forall u\in D(A_{d_2}).$$
Then \cite[Chapter 8, Lemma 2.1]{Pazy} implies that $A_{d_2}$ generates an analytic semigroup $e^{A_{d_2}t}$ on $\overline{X}$. By the essentially same arguments as in \cite[Lemmas 3.6(c)--3.7 and Theorem 5.3(a)]{HMP},
it follows that (H2) and (H3) hold true.

In the case where $\omega$ is small enough, by the arguments similar to those in \cite[Lemma 3.6(b) and Theorem 5.3(b)]{HMP}, we can also show that (H2) and (H3) are valid, and hence, system \eqref{dhmp}  has three time-space periodic solutions $E_0:=(0,0)$, $E_1:=(u^*_1(t,x),0)$ and $E_2:=(0,u^*_2(t,x))$ in $\mathbb{P}_+$.
\end{proof}
 As a consequence of Lemma \ref{M} and Theorem \ref{VLEQ}, we have the following result.
\begin{theorem}
Let (M) and either case (a) or (b) in Lemma \ref{M} hold. Then $E_1:=(u^*_1(t,x),0)$ is globally asymptotically stable for all
initial values $\phi=(\phi_1,\phi_2)\in\mathbb{P}_+$ with $\phi_1\not\equiv0$.
\end{theorem}

For simplicity, we transfer system \eqref{dhmp} into the following cooperative system:
\begin{eqnarray}\label{dhmp2}
   & &\frac{\partial u_1}{\partial t}=d_1 \frac{\partial^2 u_1}{\partial x^2}+u_1(a_\omega(t,x)-u^*_2(t,x)-u_1+u_2), \\
   & &\frac{\partial u_2}{\partial t}=d_2\frac{\partial^2 u_2}{\partial x^2}+u_1(u^*_2(t,x)-u_2)
   +u_2(a_\omega(t,x)-2u^*_2(t,x)+u_2),\quad t>0, \ x\in\mathbb{R}. \nonumber
   \end{eqnarray}
   The next result is the consequence of Theorem \ref{Qspreading}, Remark \ref{Re1} and Proposition \ref{lb}.
   \begin{theorem}
   Assume that (M) and either case (a) or (b) in Lemma \ref{M} hold. Let $u(t,\cdot,\phi)$ be the solution of system \eqref{dhmp2} with $u(0,\cdot)=\phi\in\mathcal{C}_{\beta(0,\cdot)}$. Then there exists a positive real number $\bar{c}_+$ such that the following statements are valid for system \eqref{dhmp2}:
   \begin{enumerate}
   			\item[(i)]If $\phi\in\mathcal{C}_{\beta(0,\cdot)}$, $0\le \phi\le \omega\ll \beta$ for some $\omega\in \mathcal{C}^{per}_{\beta(0,\cdot)}$, and $\phi(x)=0, \forall x\ge H$, for some $H\in \mathbb{R}$, then $\lim_{t\rightarrow\infty,x\ge ct}u(t,x,\phi)=0$ for any $c>\bar{c}_+$.
   			\item[(ii)]If $\phi\in\mathcal{C}_{\beta(0,\cdot)}$ and $\phi(x)\ge \sigma$, $\forall x\le K$, for some $\sigma\in \mathbb{R}^2$ with $\sigma\gg0$ and $K\in\mathbb{R}$, then $\lim_{t\rightarrow\infty,x\le ct}(u(t,x,\phi)-u^*(t,x))=0$ for any $c\in(0,\bar{c}_+)$.
   	\end{enumerate}
   \end{theorem}
In view of Theorem \ref{MIN}, we have the following result on periodic traveling waves for system \eqref{dhmp}.
\begin{theorem}
Let (M) and either case (a) or (b) in Lemma \ref{M} hold.
Then for any $c\ge\bar{c}_+$, system \eqref{dhmp} has time-space periodic traveling wave $(U(t,x,x-ct),V(t,x,x-ct))$ connecting $(u^*_1(t,x),0)$ to $(0,u^*_2(t,x))$ with the wave profile component $U(t,x,\xi)$  being continuous and non-increasing in $\xi$, and $V(t,x,\xi)$ being continuous and non-decreasing in $\xi$. While for any $c\in(0,\bar{c}_+)$, system \eqref{dhmp} admits no periodic rightward traveling wave connecting $(u^*_1(t,x),0)$ to $(0,u^*_2(t,x))$.
\end{theorem}

\

\end{document}